\documentclass[10pt,a4paper]{article}

\usepackage[english]{babel}
\usepackage[utf8x]{inputenc}
\usepackage{lmodern}
\usepackage[T1]{fontenc}
\usepackage{lmodern} 

\usepackage[a4paper,top=3cm,bottom=2cm,left=3cm,right=3cm,marginparwidth=1.75cm]{geometry}

\usepackage{amsmath}
\usepackage{amsthm}
\usepackage{thmtools}
\usepackage{thm-restate}
\usepackage{graphicx}
\usepackage[hidelinks]{hyperref} 
\usepackage{amscd}     
\usepackage{amsfonts}
\usepackage{dsfont}
\usepackage{latexsym}
\usepackage{amscd}
\usepackage{verbatim}
\usepackage{appendix}
\usepackage{subfigure}

\usepackage{amsmath}
	\numberwithin{equation}{section}
\usepackage{amssymb}
\usepackage{amsthm}
	\theoremstyle{plain}
		\newtheorem{thm}{Theorem}[section]
		\newtheorem{coro}[thm]{Corollary}
		\newtheorem{lem}[thm]{Lemma}
		\newtheorem{prop}[thm]{Proposition}
	\theoremstyle{definition}

	\theoremstyle{remark}
		\newtheorem{rmk}[thm]{Remark}

\usepackage{mathtools} 
\usepackage{mathrsfs} 
\usepackage{eucal} 

\usepackage{microtype}

\usepackage{mparhack}
	
\usepackage{xspace}

\usepackage{caption}
\usepackage{varioref}		
\usepackage{enumitem}
\usepackage{tikz} 
\usepackage{tikz-cd}

\newcommand{\N}{\mathbb{N}}
\newcommand{\Z}{\mathbb{Z}}
\newcommand{\K}{\mathbb{K}}

\newcommand{\F}{\mathbb{F}}

\renewcommand{\=}{\coloneqq}	 

\DeclareMathOperator{\im}{im}	
\DeclareMathOperator{\id}{id}
	
\DeclareMathOperator{\coker}{coker}	
		
\DeclareMathOperator{\mer}{mer}	
\DeclareMathOperator{\Tor}{Tor}

\title{Morse inequalities for the Koszul complex of multi-persistence}
\author{Andrea Guidolin\footnote{KTH Stockholm, Sweden}\\ \texttt{guidolin@kth.se}
\and Claudia Landi\footnote{University of Modena and Reggio Emilia, Italy}\\
\texttt{claudia.landi@unimore.it}}

\date{}

\begin{document}
\maketitle

\begin{abstract}
In this paper, we define the homological Morse numbers of a filtered cell complex  in terms of relative homology of nested filtration pieces, and derive inequalities relating these numbers to the Betti tables of the multi-parameter persistence modules of the considered filtration. Using the Mayer-Vietoris spectral sequence we first obtain strong and weak Morse inequalities involving the above quantities, and then we improve the weak inequalities achieving a sharp lower bound for homological Morse numbers. Furthermore, we prove a sharp upper bound for homological Morse numbers, expressed again in terms of the Betti tables.
\\


 
\end{abstract}

\section{Introduction}
Topological data analysis \cite{carlsson2009topology} relies on algebraic topology to extract complex information from data, with the aim of recognizing complicated patterns, inferring the topological structure underlying the data or, more generally, complementing and enhancing standard techniques in data analysis. Persistent homology, one of the most successful methods of topological data analysis, summarizes the topology of the data at multiple scales, controlled by one measured parameter, and produces informative signatures based on homology \cite{oudot2015persistence}. As understanding complex correlations in multivariate data is one of the general goals of data analysis, the development of a multivariate version of persistent homology, called multi-parameter persistence or multi-persistence \cite{carlsson2009theory} is drawing increasing interest.

In persistent homology and multi-parameter persistence, algebraic objects called persistence modules are used to encode the data in a suitable way for the extraction of topological invariants.
Typically, persistence modules are obtained from datasets endowed with measurements via an intermediate step, in which a combinatorial or topological object (e.g., a simplicial complex or, more generally, a cell complex) is constructed from the data points. As data are discrete, without loss of generality from the algebraic point of view, we can use elements of $\Z^n$ ($n\ge 1$) to encode the values of  $n$ measurements. Filtering the given data according to increasing values of the measurements (in the coordinate-wise partial order of $\Z^n$), and applying homology, one obtains the corresponding $n$-parameter persistence module \cite{oudot2015persistence}. 

In the literature, the study of persistence modules has been tackled mainly from two different perspectives. In the first works appeared in the literature about persistence theory, a dataset was encoded as a manifold, and the measurements on it as a Morse function on it. In this perspective, persistence can be grounded in Morse theory, with persistence modules determined by pairs of critical points of functions that give  {\em birth} and {\em death} to a topological feature \cite{Barannikov94,frosini1996connections,Robins2000, Edelsbrunner2002}. This paradigm has also a combinatorial counterpart developed for accelerating computations and based on discrete Morse theory  \cite{mischaikow2013morse, allili2017reducing}. In the combinatorial setting the role of critical points is played by critical cells.

From a different standpoint, it was soon realized that the persistent homology of a filtered finite simplicial  
complex is simply  a particular graded module over a polynomial ring \cite{zomorodian2005computing}. Therefore, forgetting about the data and the measurements that yield  them, persistence modules  can be studied using  tools of commutative algebra. In particular,  new invariants for persistence modules can be obtained by considering the Betti tables\footnote{In commutative algebra Betti tables are often called  (multi-graded) Betti numbers, but we prefer avoiding calling them so to avoid confusion with the Betti numbers  of persistent homology groups.} of a minimal 
free resolution of theirs \cite{carlsson2009theory}. There is, however, an important difference between the study of $n$-parameter persistence modules and the classical commutative algebra approach to graded modules: $n$-parameter persistence modules are usually obtained as the homology of a cell complex associated with data, and the connection between the algebraic object and the underlying cell complex is part of the investigation. The focus of the present article is precisely on the relation between the invariants of $n$-parameter persistence modules and the filtered cell complex from which it is obtained.

When the number of filtering parameters $n$ is equal to 1, this dual perspective is easily interpreted by noticing that births are captured by the 0th Betti table  and deaths by the 1st Betti table of  persistence modules. Persistence theory is thus a canonical way of pairing births and deaths.

For $n\ge 2$ things get more complicated as  there is no way of paring births and deaths in a natural way, and Betti tables do not mirror entrance of critical cells in the filtration. In \cite{knudson2008refinement}, this is heuristically explained  by the presence of {\em virtual} critical cells.

Borrowing from the terminology of discrete Morse theory,  where  Morse numbers are defined as the number of critical cells in the various dimensions, we introduce the notion of a {\em homological Morse number} (see below and Section~\ref{subsec:multipers}) to mean the dimension of the homology of a piece of the filtration relative to the union of all the pieces entered before. 

The main goal of the present paper is to provide insights on the interplay between the values of the Betti tables of a persistence module in any number $n\ge 1$ of parameters, and the homological Morse numbers of an $n$-filtration inducing it. 

In order to do so, we consider persistence modules obtained  applying the $q$th homology functor to an $n$-parameter filtration $\{ X^u\}_{u\in \Z^n}$ of a finite cell complex $X$. 
The corresponding $p$th Betti table $\xi_p^q\colon\Z^n\to \N$ is  obtained as the $p$th homology of the Koszul complex associated with the persistence module, a strategy already used in \cite{knudson2008refinement} and later in \cite{lesnick2019computing}. 
As for the homological Morse numbers $\mu_q(u)$ of degree $q$ at grade $u\in\Z^n$ of the filtration, we define them as the dimension of the homology at grade $u$ relative to the previous grades:  $\mu_q(u):=\dim H_q(X^u,\cup_jX^{u-e_j})$. 
Informally, $\mu_q(u)$ is the number of critical cells of $X^u-\cup_jX^{u-e_j}$ whose entrance is reflected in a change of $q$th homology. More formally, the homological Morse number $\mu_q(u)$ can be seen as the ``natural'' lower bound, in terms of homology, for the number of critical cells of dimension $q$ entering at grade $u$ (see Section~\ref{subsec:multipers}).

Under these assumptions, we present inequalities relating the homological Morse numbers $\mu_q(u)$ of a multi-filtration of a cell complex and the homology invariants of the Koszul complex of the persistence module obtained from it. In other words, we obtain Morse-type inequalities for multi-parameter persistence.

We start with the {\em strong Morse inequalities} according to which  an alternating sum  of the entries of the Betti tables of the persistence modules of $X$  at grade $u\in \Z^n$ (where the summation is on the homology degrees while the filtration grade is fixed) is bounded from above by  an alternating sum of the homological Morse numbers of the filtration of $X$ at grade $u$. This is Theorem \ref{thm:strong_relative_xi}:\\

\noindent {\bf Theorem.} For each $q\ge 0$, and each fixed grade $u\in \Z^n$, we have
\[
\sum_{i=0}^q (-1)^{q+i} \mu_i (u)  \ge \sum_{i=0}^q (-1)^{q+i} \left( \xi_0^i(u) - \sum_{p=1}^{i+1}\xi_p^{i+1-p}(u)    \right) .
\]
\\

As usual, the strong inequalities imply the {\em weak Morse inequalities}, 
given in Corollary \ref{coro:weak_relative_xi}:\\

\noindent{\bf Corollary.}
For each $q\ge 0$, and each fixed grade $u\in\Z^n$, we have
\[
\mu_q (u) \ge   \xi_0^{q}(u) - \sum_{p=1}^{q+1}\xi_p^{q+1-p}(u).
\]

Moreover, we can define the {\em Euler characteristic} of a filtration at grade $u$ by considering the relative homology of $(X^u,\cup_j X^{u-e_j})$ and setting
 $$\chi (X^u,\cup_j X^{u-e_j}):=\sum_{q} (-1)^q \dim H_q (X^u,\cup_j X^{u-e_j}). $$
It is interesting to see how this notion of Euler characteristic of a filtration relates to the Betti tables of its persistence modules. As shown in Theorem \ref{thm:chi_relative_betti},  for $q$ large enough, the $q$th strong Morse inequality is actually an equality:\\
 
\noindent{\bf Theorem.}
$\chi (X^u,\cup_j X^{u-e_j}) = \sum_{i=0}^{\dim X+1} (-1)^i \sum_{p=0}^i \xi_p^{i-p}(u)$.\\

The weak Morse  inequalities of Corollary \ref{coro:weak_relative_xi} are too weak to be also sharp. In order to achieve sharpness, we improve them by proving  Theorem \ref{thm:lower}, which  can be summarized as follows.\\ 

 \noindent{\bf Theorem.}
For an $n$-parameter filtration $\{ X^u \}_{u\in \Z^n}$, for each grade $u\in \Z^n$, and for each $q\ge 0$, 
\begin{equation*}
\mu_q (u) \ge \xi^q_0 (u) + \xi^{q-1}_1 (u) - \sum_{p=1}^{n-1} \xi_{p+1}^{q-p} (u)+R, 
\end{equation*}
where
$R$ is a non-negative integer.\\
 
Examples are provided to show that these estimates are sharp for every number of parameters $n$. In particular, these inequalities show that a non-trivial $p$th Betti table at $u$  does not need the entrance of a critical cell at $u$ as the presence of positive and negative terms in the right-hand side compensate each other. This is different from what happens when $n=1$. Indeed, in the case of a single parameter, the above inequalities reduce to $\mu_q(u)\ge \xi^q_0 (u)+ \xi^{q-1}_1 (u)$. Since in one-parameter persistence $\xi^q_0$ counts the number of persistence births  in  homology of degree $q$, and  $\xi^{q-1}_1$ counts the number of persistence death in homology of degree $q-1$,  these inequalities say that, when $n=1$, in order to have a birth or a death we necessarily need the entrance of a critical cell, as is well known.

On the other hand, in Theorem \ref{thm:upper} we also present  the inverse inequalities,  proven to be sharp as well, showing that 
a non-zero homological Morse number
at grade $u$ of the filtration necessarily causes some Betti table to become non-trivial:\\

\noindent{\bf Theorem.} For an $n$-parameter filtration $\{ X^u \}_{u\in \Z^n}$, for each grade $u\in \Z^n$, and for each $q\ge 0$, we have
\[ \mu_q(u)  \le \sum_{i=0}^n \xi_i^{q-i}(u) .
\]
We observe that when $n=1$ these new bounds reduce to $\mu_q(u)\le \xi^q_0 (u)+ \xi^{q-1}_1 (u)$. Here it is important to underline that these inequalities hold thanks to the fact that we are considering homological Morse numbers and not the numbers of all critical cells entering at $u$.

All the inequalities provided in this paper are obtained via the Mayer-Vietoris spectral sequence associated with a double complex built from the filtration of $X$. This strategy generalizes that of \cite{landi2021discrete} where the particular case of $n=2$ is studied applying the Mayer-Vietoris homology exact sequence. In particular, differently from papers like \cite{Lipsky2011,govc2018approximate,casas2019distributing}, where the Mayer-Vietoris spectral sequence is used in the context of single-parameter persistent homology to merge local data into global information, we work only locally at a fixed grade $u$ of the multi-parameter filtration, but considering all the possible homology degrees.

At a basic level, our inequalities prove that a persistence module  having \lq\lq large'' Betti tables at grade $u$ does not necessarily come from a filtration with a large number of critical cells entering at $u$ unless the $p$th Betti tables with $p\ge 2$ are trivial. On the other hand, a large homological Morse number necessarily implies large values in the Betti tables of specific indices. 

This is a more diversified behavior than the case $n=1$, in which we have $\mu_q(u)=\xi^q_0 (u)+ \xi^{q-1}_1 (u)$ by combining the inequalities introduced above. In other words, the homological Morse number of degree $q$ is equal to the number of births in  degree $q$ plus the number of deaths in  degree $q-1$.

At a higher level, we believe our results are interesting from different perspectives. Firstly, in contrast to the state of the art in the literature where multi-parameter persistence modules are usually studied as obtained by applying homology in an arbitrary but fixed degree, our inequalities show the interplay among persistent homology modules at all the various homology degrees simultaneously. In particular, we can see that, contrary to one-parameter persistence where the entrance of a critical cell of dimension $q$ can modify only the persistence modules in degree $q$ or $q-1$, for multi-parameter persistence the effect of its entrance may involve many more homology degrees.

Secondly, starting from Knudson's observation  in \cite{knudson2008refinement} about the fact that the Betti tables in multi-parameter persistence are determined not only by elements corresponding to real critical cells in the filtration, but also by elements corresponding to virtual cells, our inequalities allow for a  measure of the gap between the number of real critical cells and that of virtual critical cells.

\bigskip
\noindent {\em Organization of the article.} 
In Section \ref{sec:prelim} we review the needed background on cell complexes, Morse inequalities and persistence modules, and provide a brief description of the Mayer-Vietoris spectral sequence.
In Section \ref{sec:koszul_cmp} we describe the Koszul complex of a multi-parameter persistence module and its Betti tables.
In Section \ref{sec:setupMV} we introduce the Mayer-Vietoris spectral sequence of a multi-parameter filtration and show its relation with the Betti tables. 
In Section \ref{sec:strong} we derive Morse inequalities for multi-parameter persistence modules, which are applied in Section \ref{sec:euler} to obtain Euler characteristic formulas for the relative homology of the filtration.
In Section \ref{sec:improve} we improve our Morse inequalities using the Mayer-Vietoris spectral sequence and show the sharpness of our lower and upper bound for homological Morse numbers in terms of the Betti tables.

\section{Preliminaries}
\label{sec:prelim}

\subsection{Chain complexes, cell complexes and homology}
\label{subsec:chaincm}

Let $\F$ be a fixed field. In this work, we consider bounded finitely generated chain complexes $C_* = (C_q,\partial_q)_{q\in \Z}$ over $\F$, simply called  {\em chain complexes}, meaning that $C_q =0$ whenever $q<0$ or $q\ge m$, for some $m\in \N$, and each $C_q$ is a finite dimensional vector space over $\F$. Let us further assume that a distinguished (finite) $\F$-basis $X_q$ of each $C_q$ is given, so that  $C_q \cong \bigoplus_{\sigma \in X_q} \F \sigma$. A chain complex $C_*$ endowed with such distinguished bases is called a \emph{based chain complex}. We use the notation $C_* (X) = (C_q(X),\partial_q)_{q\in \Z}$ to explicitly recall the fixed bases $X$ of $C_*$. We can express the differentials $\partial_q : C_q(X) \to C_{q-1}(X)$ with respect to the fixed bases as 
\[ \partial_q (\tau) = \sum_{\sigma\in X_{q-1}} \kappa (\tau ,\sigma) \sigma
\]
for each $\tau\in X_q$;
in other words, for each $\tau \in X_q$ and $\sigma \in X_{q-1}$, we denote by $\kappa (\tau ,\sigma)$ the coefficient with which $\sigma$ appears in $\partial_q(\tau)$.

The distinguished bases of $C_*$ inherit a combinatorial structure which coincides with the abstract notion of a cell complex  as introduced by Lefschetz \cite{lefschetz1942algebraic} (see also \cite{Harker2014}).
In topological data analysis, considering this equivalent combinatorial perspective is sometimes advantageous, since a cell complex is usually constructed from the data and hence interpretable in the concrete situation at hand.
A \emph{cell complex} is a finite graded set $X=\bigsqcup_{q\in \Z} X_q$, whose elements are called \emph{cells}, endowed with an \emph{incidence function} $\kappa : X\times X \to \F$. A cell $\sigma \in X_q$ is called a $q$-\emph{cell} or a cell of \emph{dimension} $q$, denoted $\dim \sigma = q$. The \emph{dimension of} $X$ is defined as the maximum dimension of its cells. The incidence function must satisfy the following conditions: 
(\textit{i}) $\kappa (\tau , \sigma) \ne 0$ implies $\dim \tau = \dim \sigma + 1$, and 
(\textit{ii}) for each $\tau$ and $\sigma$ in $X$, it holds $\sum_{\rho\in X}\kappa (\tau , \rho)\cdot \kappa (\rho ,\sigma)=0$. We regard $X$ as the graded poset endowed with the partial order $\le$ generated by the covering relation $\sigma < \tau$ whenever $\kappa (\tau , \sigma) \ne 0$.

We underline that based chain complexes constitute a rather general setting, since chain complexes canonically associated with many combinatorial or topological objects (such as simplicial complexes, cubical complexes, finite CW complexes) fall within this definition.

As an example, an \emph{abstract simplicial complex} $\Delta$ given by  a  collection of non-empty finite  subsets of a
given set $S$, with the property of being  closed under taking subsets, can be regarded as a cell complex as follows: each $\sigma\in \Delta$ containing $q+1$ elements can be viewed as a $q$-cell, and called a  $q$-\emph{simplex}, and in particular singletons are called vertices.  Fixing an ordering for vertices induces an ordering on the elements of each simplex, and one can define the incidence function 
\[ \kappa (\tau , \sigma)  \= \left\{
       \begin{array}{ll}
            (-1)^i & \quad \text{if $\sigma$ is obtained from $\tau$  by removing its $(i+1)$th element} \\
            0 & \quad \text{otherwise}
        \end{array}
    \right.
\]
which induces the usual simplicial boundary map.

A collection of subsets $A_q \subseteq X_q$ freely generates a chain subcomplex $C_* (A) \subseteq C_* (X)$ if and only if $A=\bigsqcup_{q\in \Z} A_q$ is a \emph{subcomplex} of $X$, meaning that, endowed with the restriction of the incidence function of $X$, it is a cell complex in its own right.    Given a cell  complex $X$ and a subcomplex $A\subseteq X$, the relative chain complex $C_*(X,A)$ is defined as the chain complex $(C_q(X)/C_q(A),\partial'_q)_{q\in \Z}$, with $\partial'_q$ being the differential induced by $\partial_q$ on the quotient. 

Applying $q$th homology to a chain complex $C_*$ gives the $\F$-module $H_q(C_*) = \ker \partial_q / \im \partial_{q+1}$,  denoted $H_q(X)$ if the chain complex has a distinguished basis $X$. Analogously, the notation $H_q(X,A)$ is used for homology of a relative chain complex $C_*(X,A)$. In this paper, homology is always assumed to be over a fixed field $\F$, so that taking homology or relative homology of a complex always gives (finite-dimensional) $\F$-vector spaces. 

\subsection{Standard Morse inequalities}
\label{sec:Morse-inequalities}

Given a (non-negatively graded) chain complex $C_* = (C_q , \partial_q)$ and setting $c_q \= \dim C_q$, the \emph{strong Morse inequalities} are: 
\begin{equation}
\label{eq:strongCd}
\sum_{i=0}^q(-1)^{q+i}c_i \ge \sum_{i=0}^q(-1)^{q+i} \dim H_i(C_*),
\end{equation}
for all $q\ge 0$. These inequalities are obtained via standard linear algebra by observing that the equations 
\[
\dim H_i(C_*) = \dim \ker \partial_{i} - \dim \im \partial_{i+1} = c_i - \dim \im \partial_{i} - \dim \im \partial_{i+1}, 
\]
for all $i$, imply that the difference between the left-hand side and the right-hand side of (\ref{eq:strongCd}) is $\dim \im \partial_{q+1}$.
Strong inequalities imply  \emph{weak Morse inequalities}: 
$c_q \ge \dim H_q(C_*)$ for all $q\ge 0$.
As is well known, they are obtained simply by observing that $c_q = \sum_{i=0}^q (-1)^{q+i}c_i + \sum_{i=0}^{q-1} (-1)^{q-1+i}c_i$ and applying the corresponding strong inequalities.

Moreover, if $C_*$ is bounded, for values of $q$ sufficiently large the strong inequalities are actually equalities involving the Euler characteristic $\chi(C_*)\=\sum_{q\ge 0}(-1)^q c_q$ of $C_*$: it holds that
\[
\chi (C_*) =  \sum_{q\ge 0}(-1)^q \dim H_q(C_*).
\]

Weak Morse inequalities represent constraints on the number of generators of a chain complex $C_*$, which can be improved by replacing $C_*$ with a chain complex quasi-isomorphic to it with less generators.
This strategy is used, for example, in \cite{Forman1998} where,  endowing a regular cell complex  $X$ with a discrete Morse function $f$, $C_*(X)$ is shown to be quasi-isomorphic to  the Morse complex containing only the  critical cells of $f$. Thus, $c_q$ can be taken to  coincide with the number of  critical cells of $f$ with dimension $q$. Similarly, in the case of a PL Morse function defined on a simplicial complex, strong and weak inequalities hold with $c_q$ being the number of critical vertices of index $q$ \cite{edelsbrunner2010computational}.

\subsection{Multi-filtrations and multi-parameter persistence}
\label{subsec:multipers}

Persistent homology was originally introduced as a method to encode in a single object the evolution of the homology of a family of nested cell complexes (usually simplicial complexes) parametrized by a linearly ordered set of indexes, such as the integers or the reals \cite{Barannikov94,frosini1996connections,Robins2000,Edelsbrunner2002}. Later it  became clear that families of nested complexes parametrized over other sets of indices can be equally relevant (see, e.g., \cite{oudot2015persistence}  for a review). In particular, \emph{multi-persistence} \cite{carlsson2009theory} treats the case of integer  parameters along multiple directions, that is a grid. This is the setting we consider here.

For an integer $n\ge 1$, indicating the grid dimension, we denote by 
$[n]$ the set $\{1, 2, \ldots ,n\}$, by $\{e_j\}_{j\in [n]}$ the standard basis of $\Z^n$, and by  $\preceq$ the coordinate-wise partial order on $\Z^n$: if $u=(u_i),v=(v_i)\in \Z^n$, we write $u\preceq v$ if and only if $u_i \le v_i$, for all $1\le i\le n$.

An $n$-\emph{parameter persistence module} $V$ consists in a collection $\{ V^u \}_{u\in \Z^n}$ of $\F$-vector spaces and a collection $\{ \varphi^{u,v} : V^u \to V^v \}_{u\preceq v \in \Z^n}$ of linear maps such that $\varphi^{u,w} = \varphi^{v,w} \circ \varphi^{u,v}$ whenever $u\preceq v\preceq w$, and $\varphi^{u,u} = \id_{V^u}$, for all $u$. 

In applications, persistence modules usually originate from filtrations of cell complexes. An $n$-\emph{filtration} of a complex $X$ is a family $\mathcal{X}= \{ X^u \}_{u \in \Z^n}$ of subcomplexes of $X$ such that $X^u \subseteq X^v$ whenever $u\preceq v$. If $n>1$ we refer to $\mathcal{X} = \{ X^u \}_{u \in \Z^n}$ generically as a \emph{multi-filtration}, as opposed to the case $n=1$ that is called simply a (single-parameter) filtration. The index $u\in\Z^n$ is called a \emph{filtration grade}. If $\sigma \in X^u - \bigcup_{j=1}^n X^{u-e_j}$, then $u$ is called an {\em entrance grade} of $\sigma$ in $\mathcal{X}$. The \emph{dimension} of $\mathcal{X}$ is, by definition, the dimension of $X$.

Throughout this article, we make the important assumption that all filtrations $\mathcal{X}= \{ X^u \}_{u \in \Z^n}$ we consider arise from the sublevel sets of an order-preserving function $f:(X,\le) \to (\Z^n ,\preceq)$, where the partial order on $X$ is defined in Section \ref{subsec:chaincm}. This means that $X^u = \{ \sigma \in X \mid f(\sigma) \preceq u \}$, for every $u\in \Z^n$. 
In topological data analysis, such a filtration of a cell complex $X$ is usually called \emph{one-critical} \cite{carlsson2009computing} because every cell $\sigma\in X$ admits exactly one entrance grade. 
In what follows, this assumption on filtrations will be crucial, as it ensures that, for each subset $\sigma\subseteq [n]$, setting $e_\sigma=\sum_{j\in\sigma}e_j$, we have 
\begin{equation}
\label{eq:one-crit-consequence}
\bigcap_{j\in \sigma} X^{u-e_j} = X^{u-e_\sigma} .
\end{equation}
Additionally, we assume all filtrations to be \emph{bounded} by requiring that $X^u\ne\emptyset$ implies $0\preceq u$, and that $X^u =X$ whenever $u$ is sufficiently large.

Applying the $q$th homology functor to an  $n$-filtration $\mathcal{X} = \{ X^u \}_{u \in \Z^n}$ yields the $n$-\emph{parameter persistent homology module} ${V}_q=\{ {V}_q^u, \iota^{u,v}_{q}\}_{u\preceq v\in \Z^n}$, with ${V}_q^u =H_q (X^u)$ and 
$\iota^{u,v}_{q} \colon H_q(X^u) \to H_q (X^v)$ induced by the inclusion maps $X^u \hookrightarrow X^v$ for $u\preceq v$. 
We denote this persistence module by $V_q = H_q (\mathcal{X})$.
Inspired by the one-parameter situation where a critical filtration grade is characterized by the property that the relative homology of the pair $(X^u, X^{u-1})$ is  non-trivial  (cf., e.g., \cite{FugacciLV20}), a grade $u\in\Z^n$ of a multi-filtration $\mathcal{X} = \{ X^u \}_{u \in \Z^n}$ will be said to be a {\em critical filtration grade of index $q$} if $H_q(X^u, \cup_j X^{u-e_j})$ is non-trivial.
Clearly, critical filtration grades are a subset of entrance grades. We call the number
\begin{equation}\label{def:relative-c_q}
\mu_q(u)\=\dim H_q(X^u, \cup_{j=1}^n X^{u-e_j})
\end{equation}
the \emph{homological Morse number} of degree $q$ at $u\in \Z^n$. Informally, we can think of $\mu_q(u)$ as the number of critical cells whose entrance at $u\in \Z^n$ has an effect on $q$th homology of the filtration. The connection with the number of critical cells with entrance grade $u$, as defined in discrete Morse theory \cite{Forman1998,kozlov2005discrete} (with the notion of \emph{critical} depending on the choice of a discrete gradient vector field, also called an acyclic matching), can be made rigorous. Let $m_q(u)$  denote the $q$th \emph{Morse number} at $u$ of the filtration $\mathcal{X}$, defined as the number of critical $q$-cells with entrance grade $u$; then $m_q(u)\ge \mu_q(u)$ \cite[Prop.~1]{landi2021discrete}, for every choice of an acyclic matching used to define $m_q(u)$.

\subsection{The Mayer-Vietoris spectral sequence}
\label{subsec:MV_preliminaries}

In this subsection we provide a brief description of the Mayer-Vietoris spectral sequence. We follow \cite[Ch.~VII]{Brown1982} and adapt the construction to the case of cell complexes. 
The Mayer-Vietoris spectral sequence is a particular case of a spectral sequence associated with a double complex, a standard construction that can be found in most books on homological algebra (see for example \cite{weibel1994introduction,rotman2009homological,maclane2012homology}). More details on the Mayer-Vietoris spectral sequence in our situation of interest are given in Section \ref{sec:setupMV}. 

Let $X$ be a cell complex (as defined in Section \ref{subsec:chaincm}) and let $\{ X^j \}_{j\in J}$ be a collection of subcomplexes of $X$, with $J$ a totally ordered index set.
The Mayer-Vietoris spectral sequence relates the homology of the union of the collection $\cup_{j\in J}X^j$ with the homology of the subcomplexes $X^j$ and their intersections $\cap_{j\in \sigma}X^j$ for $\sigma \subseteq J$.
The \emph{nerve} $\Sigma$ of the collection $\{ X^j \}_{j\in J}$ is defined as the abstract simplicial complex of all $\sigma \subseteq J$ such that $\cap_{j\in \sigma} X^j \ne \emptyset$. 
For all $p\ge 0$ we denote $\Sigma_p$ the set of $p$-simplices of $\Sigma$, which are of the form $\sigma =\{ j_1 < \cdots < j_{p+1} \}$.
For all $p$, consider the chain complexes
\[ C_{p,*} =  \bigoplus_{\sigma \in \Sigma_p} C_* (\cap_{j \in \sigma} X^j) \]
with differential maps 
\begin{equation}
\label{eq:horiz_d}
\delta_{p,*} :C_{p,*} \to C_{p-1,*} 
\end{equation} 
between them defined as follows: for $1\le \ell \le p+1$ and  $\sigma =\{ j_1 < \cdots < j_{p+1} \}$ consider  $\partial_\ell \sigma \=\{ j_1 , \ldots , \hat{\jmath}_\ell , \ldots , j_{p+1} \}$, obtained by removing $j_\ell$; then, for $p\ge 1$, observe that the inclusions $C_*(\cap_{j\in \sigma}X^{j})\hookrightarrow C_*(\cap_{j\in \partial_\ell \sigma}X^{j})$ induce chain maps $\delta^{(\ell)}_{p,*} :C_{p,*} \to C_{p-1,*}$ and define $\delta_{p,*} \= \sum_{i=0}^p (-1)^{i} \delta_{p,*}^{(p+1-i)}$.
A chain map $\varepsilon_* = \delta_{0,*}:C_{0,*} \to C_*(\cup_{j\in J} X^j)$ is induced in a similar way by the inclusions $C_*(X^{j})\hookrightarrow C_*(\cup_{j\in J} X^{j})$.
To facilitate our manipulations in the following sections, in the definition of $\delta_{p,*}$ we made a different choice from \cite{Brown1982} regarding the alternating signs, which leads however to an isomorphic construction of the double complex.

The following sequence of chain complexes is exact (see \cite{Brown1982}):
\begin{equation}
\label{eq:augm}
0 \xleftarrow{} C_*(\cup_{j\in J} X^j) \xleftarrow{\varepsilon_*} C_{0,*} \xleftarrow{\delta_{1,*}}  C_{1,*} \xleftarrow{} \cdots \xleftarrow{} C_{p-1,*} \xleftarrow{\delta_{p,*}} C_{p,*} \xleftarrow{} \cdots
\end{equation}
We note that if the index set $J$ is finite and $m = |J|$, then $C_{p,*}=0$ for all $p\ge m$. 
We will henceforth refer to the sequence of chain complexes $C_{0,*} \xleftarrow{\delta_{1,*}} \cdots  \xleftarrow{\delta_{p,*}} C_{p,*} \xleftarrow{} \cdots$ as the \emph{truncation} of the exact sequence (\ref{eq:augm}). As a consequence of the definitions, this is a double complex
\begin{equation*}
\label{eq:MVdoublecmp}
C_{p,q} =  \bigoplus_{\sigma \in \Sigma_p} C_q (\cap_{j \in \sigma} X^j)
\end{equation*} 
with the horizontal differential $\delta_{p,q}: C_{p,q} \to C_{p-1,q}$ we just introduced, and the vertical differential $\partial_{p,q}: C_{p,q} \to C_{p,q-1}$ induced by the differential of $C_*(X)$, with a sign change of $(-1)^p$ to ensure that squares are anticommutative (that is, $\partial_{p-1,q}\delta_{p,q}+\delta_{p,q-1}\partial_{p,q}=0$), which is the convention we choose for double complexes in this article.

Let $T_*$ denote the total complex of the double complex $\{ C_{p,q}, \delta_{p,q}, \partial_{p,q} \}_ {p,q \in \Z}$, which is the chain complex with chain groups $T_k$ and differentials $d^T_k$ are defined by 
\[ T_k \= \bigoplus_{p+q=k} C_{p,q}, \qquad
 d^T_k \= \sum_{p+q=k} (\delta_{p,q} +\partial_{p,q}) .
\]
The \emph{Mayer-Vietoris spectral sequence} of the collection $\{ X^j \}_ {j\in J}$ is defined as the spectral sequence associated with the \emph{first filtration} $F^{\text{I}}_p$ of $T_*$ (see e.g. \cite[Ch.~10]{rotman2009homological} for details), given by $F^\text{I}_p T_k \= \bigoplus_{i\le p} C_{i,k-i}$, with differentials induced by $d^T$. 
The Mayer-Vietoris spectral sequence converges to the homology $H_* (\cup_{j\in J} X^j)$ of the union $\cup_{j\in J} X^j$, since it can be shown \cite{Brown1982} that $H_*(T_*) \cong H_* (\cup_{j\in J} X^j)$.

\section{The Koszul complex of persistence and its Betti tables}
\label{sec:koszul_cmp}

The direct analysis of an $n$-parameter persistence module when $n>1$ is in general quite complicated due to the lack of a finite or at least tame family of indecomposable summands for such objects, as proved by Gabriel in  \cite{gabriel}.  Hence, one often resorts to simpler albeit incomplete algebraic invariants of a persistence module. In this paper we focus on the Betti tables (also called multi-graded Betti numbers) of a persistence module, calculated via the homology of its Koszul complex.  

Betti tables have been studied since early works on multi-parameter persistence \cite{carlsson2009theory,knudson2008refinement} where it was noted that there is an equivalence  between the category of $n$-parameter persistence modules and the category of $n$-graded modules over the polynomial ring $S:=\F [x_1, \ldots ,x_n]$. 
Explicitly, the correspondence takes a persistence module $\{ V^u ,\varphi^{u,v} \}$ to the $n$-graded $S$-module $\bigoplus_{u\in \Z^n} V^u$ with the action of $S$ defined by $x_i \cdot z = \varphi^{u,u+e_i} (z)$, for all $z\in V^u$ and all $i\in [n]$. This correspondence allows for the use of tools from commutative algebra to study persistence modules. We refer the reader to \cite{miller2005combinatorial} for background on such invariants for $n$-graded modules, while here we adopt the point of view of persistence modules.

Given an $n$-parameter persistence module $V$,  and regarding it as an $n$-graded $S$-module $V = \bigoplus_{u\in \Z^n}V^u$ via the equivalence of categories  mentioned above,  the $i$th \emph{Betti table}  (or \emph{multi-graded Betti numbers})  of $V$ is defined as $\xi_i\colon\Z^n\to \N$ with
\[ \xi_i (u) = \dim_{\F} (\Tor_i^{S}(V,\F )(u)),
\]
for all $u\in \Z^n$ and all  $i\in \{ 0,1, \ldots ,n\}$, where $\Tor_i^{S}(V,\F )(u)$ is the part of grade $u$ of  $\Tor_i^{S}(V,\F )$ viewed as an $n$-graded $S$-module.
By definition of $\Tor_i^{S}(V,\F )$, the $i$th Betti table $\xi_i (u)$ of $V$ at $u$ can thus be calculated by applying the functor $-\otimes_S \F$ to a free resolution of $V$, taking $i$th homology of the resulting chain complex and considering the dimension over $\F$ of the part of grade $u$ of the homology module.

The general property $\Tor_i^{S}(V,\F ) \cong \Tor_i^{S}(\F, V)$ from homological algebra (see, e.g., \cite[Theorem 7.1]{rotman2009homological}) provides an alternative way to calculate the $i$th Betti table of $V$ by applying the functor $- \otimes_S V$ to a free resolution of $\F$ and taking $i$th homology. This yields an equivalent definition of the Betti tables of $V$ based on its Koszul complex.
Given an $n$-graded $S$-module $V=\bigoplus_{u\in \Z^n}V^u$, the \emph{Koszul complex}  \emph{of $V$ at grade $u\in\Z^n$}, denoted $\mathbb{K}_*(x_1,\ldots ,x_n; V)(u)$, is the part of grade $u$ of the ($n$-graded) chain complex $\mathbb{K}_* \otimes_S V$, where  $\mathbb{K}_*= \mathbb{K}_*(x_1,\ldots ,x_n)$ is the classical Koszul complex of $S$, defined for example in \cite[Def. 1.26]{miller2005combinatorial} or \cite[Ch. 17.2]{eisenbud1995commutative}. 
Below, we provide an explicit definition of $\mathbb{K}_*(x_1,\ldots ,x_n; V)(u)$. Since $\mathbb{K}_*$ is a (minimal) free resolution of $\F \cong S/\langle x_1, \ldots ,x_n\rangle$ \cite[Prop. 1.28]{miller2005combinatorial}, as observed above the $i$th homology module of the chain complex $\mathbb{K}_* \otimes_S V$ has dimensions (over $\F$) in the various grades $u\in \Z^n$ coinciding with the Betti table $\xi_i$ of $V$. In other words, for each $u\in \Z^n$,
\[
\xi_i (u) = \dim_{\F} H_i (\mathbb{K}_*(x_1,\ldots ,x_n; V)(u)) . 
\]

For our purposes, we focus on the Betti tables $\xi_i^q$ of the persistent homology module $V_q$ arising from the $q$th homology of a filtration $\{ X^u \}_{u \in \Z^n}$. For each $i\in \{0,1,\ldots , n \}$, the  module appearing in degree $i$ in the chain complex  $\mathbb{K}_*(x_1,\ldots ,x_n; V_q)(u)$ is
\begin{equation}
\label{eq:koszul-space}
\mathbb{K}_i (x_1,\ldots ,x_n; V_q)(u) = \bigoplus_{|\sigma|=i} H_q (X^{u-e_{\sigma}})
\end{equation}
with $e_{\sigma}\= \sum_{j\in \sigma} e_j$ and $\sigma \subseteq [n]$. The modules $\mathbb{K}_i (x_1,\ldots ,x_n; V_q)(u)$ are zero for all $i\notin \{0,1,\ldots n \}$.
The differentials  of $\mathbb{K}_* (x_1,\ldots ,x_n; V_q)(u)$ are defined in terms of the maps $\iota^{v,w}_{q}: H_q (X^v) \to H_q (X^w)$ that define $V_q$ as follows: the restriction of
\begin{equation}
\label{eq:koszul-differential}
d_i : \mathbb{K}_i (x_1,\ldots ,x_n; V_q)(u) \to \mathbb{K}_{i-1} (x_1,\ldots ,x_n; V_q)(u)
\end{equation}
to each direct summand $H_q(X^{u-e_{\sigma}})$ of its domain, with  $\sigma = \{ j_1 < j_2 <\ldots < j_i \}$, is
\begin{equation*}
\label{eq:expl_def_dKosz}
d_{i|} = \sum_{r=0}^{i-1} (-1)^r \iota_q^{u-e_{\sigma},u-e_{\partial_{i-r} \sigma}} ,
\end{equation*}
where $\partial_{i-r} \sigma \=  \{ j_1, \ldots ,\hat{\jmath}_{i-r}, \ldots ,j_i \}$.
As we said, $\xi^q_i(u)$ can be defined as the dimension (over $\F$) of the $i$th homology module of $\mathbb{K}_*(x_1,\ldots ,x_n; V_q)(u)$.

Let us examine the map $d_1$ in the Koszul complex. This map, sometimes called the \emph{merge map} and denoted  $\mer_q^u$, is the composition of the available maps $\overline{\varepsilon}^u_q$ followed by $i^u_q$, making the diagram
\begin{equation}
\label{eq:def_mer}
\begin{tikzcd}
 \bigoplus_{j} H_q (X^{u-e_j})  \arrow[dr, "\mer_q^u = d_1", swap] \arrow[r, "\overline{\varepsilon}^u_q"] & H_q (\cup_j X^{u-e_j}) \arrow[d, "i^u_q"]  \\
& H_q (X^u) 
\end{tikzcd}
\end{equation}
commute. Here, $\overline{\varepsilon}^u_q$ is the map induced by the obvious inclusions, whose restriction to each direct summand $H_q (X^{u-e_j})$ of the domain is the map $H_q (X^{u-e_j}) \to H_q (\cup_j X^{u-e_j})$ induced in homology by $X^{u-e_j}\hookrightarrow \cup_j X^{u-e_j}$. The map $i^u_q$ is induced in homology by the inclusion $\cup_j X^{u-e_j} \hookrightarrow X^u$.

\section{The Mayer-Vietoris spectral sequence of a multi-filtration}
\label{sec:setupMV}

Let $\{ X^u\}_{u\in \Z^n}$ be an $n$-parameter filtration of a cell complex $X$.
For a fixed grade $u \in \Z^n$ of this  filtration, consider the collection of cell subcomplexes $\{ X^{u-e_j} \}_{j\in[n]}$.  
If $n=2$, it is well-known that there is a short exact sequence
\begin{equation*}
\label{eq:shortMV}
0 \xrightarrow{} C_* (X^{u-e_1} \cap X^{u-e_2}) \xrightarrow{} C_* (X^{u-e_1}) \oplus C_* (X^{u-e_2})  \xrightarrow{} C_* (X^{u-e_1} \cup X^{u-e_2}) \xrightarrow{} 0  
\end{equation*}
inducing in homology the Mayer-Vietoris long exact sequence, which clarifies the relation between the homology of $X^{u-e_1} \cap X^{u-e_2}$, $X^{u-e_1}$, $X^{u-e_2}$ and $X^{u-e_1} \cup X^{u-e_2}$. This can be generalized for $n>2$ via the Mayer-Vietoris spectral sequence relating the homology of $\cap_{j\in \sigma} X^{u-e_j}$, for all $\sigma \subseteq [n]$, to the homology of $\cup_{j\in [n]} X^{u-e_j}$.
Even if the Mayer-Vietoris spectral sequence can be defined for general collections of subcomplexes as seen in Section \ref{subsec:MV_preliminaries}, in this article we will focus on the collection $\{ X^{u-e_j} \}_{j\in[n]}$ for a fixed grade $u\in \Z^n$. Here, we provide more details on the Mayer-Vietoris spectral sequence associated with the collection of subcomplexes $\{ X^{u-e_j} \}_{j\in[n]}$, in preparation to describe the connection with the Koszul complex.

Given an $n$-parameter filtration $\{ X^u\}_{u\in \Z^n}$, fix a grade $u\in \Z^n$ and consider the collection of subcomplexes $\{ X^{u-e_j} \}_{j\in[n]}$. Considering the intersections of all possible subcollections of $\{ X^{u-e_j} \}$, we can define a double complex $C_{*,*} = \{ C_{p,q} , \delta_{p,q} , \partial_{p,q} \}_{p,q\in \Z}$ with
\begin{equation}
\label{eq:Cdouble} 
C_{p,q} \= 
\left\{\begin{array}{lr}
\bigoplus_{|\sigma| = p+1} C_q (\cap_{j\in \sigma} X^{u-e_j}) & \text{if } p,q \ge 0\\
0  & \text{otherwise,}
\end{array}\right. 
\end{equation}
for $\sigma \subseteq [n]$, and with the two differentials $\delta_{p,q} : C_{p,q} \to C_{p-1,q}$ and $\partial_{p,q} : C_{p,q} \to C_{p,q-1}$ defined as in Section \ref{subsec:MV_preliminaries}. The Mayer-Vietoris spectral sequence is the (first quadrant) spectral sequence associated with the filtration $\{ F^\text{I}_p T_* = \bigoplus_{i\le p} C_{i,*-i} \}_{p\in \Z}$ of the total complex $T_*$ of $C_{*,*}$, and it converges to $H_* (T_*)\cong H_* (\cup_{j\in [n]} X^{u-e_j})$. 
Let us recall (see e.g. \cite{weibel1994introduction,rotman2009homological}) that convergence of the spectral sequence to $H_*(T_*)$ is  expressed by isomorphisms
$E^{\infty}_{p,q} \cong \mathcal{F}_p H_{k}(T_*) / \mathcal{F}_{p-1} H_{k}(T_*)$, for all $p$, $q$ and $k=p+q$, where $\{ \mathcal{F}_p H_* (T_*) \}_{p\in \Z}$ is the induced filtration on $H_*(T_*)$ defined by
\begin{equation} 
\label{eq:inducedFpHT}
\mathcal{F}_p H_k (T_*) \= \im (f^p_k : H_k (F^{\text{I}}_p T_*) \to H_k (T_*) ) , 
\end{equation} 
with $f^p_k$ being the map induced by the inclusion $F^{\text{I}}_p T_* \xhookrightarrow{} T_*$. 

A key observation for this work is that, as we showed in (\ref{eq:one-crit-consequence}), the one-criticality assumption on the $n$-parameter filtration ensures that $\cap_{j\in \sigma} X^{u-e_j} = X^{u-e_{\sigma}}$, with $e_{\sigma}\= \sum_{j\in \sigma} e_j$, for each $\sigma \subseteq [ n ]$. Keeping this in mind, we want to explicitly describe the low-degree pages of the spectral sequence, as is possible for spectral sequences associated with double complexes (see e.g. \cite{weibel1994introduction,rotman2009homological}).

The $0$-page of the Mayer-Vietoris spectral sequence associated with the collection $\{ X^{u-e_j} \}_{j\in[n]}$ has terms $E^0_{p,q} = C_{p,q} = \bigoplus_{|\sigma| = p+1} C_q ( X^{u-e_{\sigma}})$ and differentials $d^0_{p,q} = \partial_{p,q}: C_{p,q} \to C_{p,q-1}$ induced by the differentials of $C_*(X)$, up to a $(-1)^p$ sign change (see Section \ref{subsec:MV_preliminaries}).
The terms of the $1$-page are therefore
\[
E^1_{p,q}=H_q (C_{p,*}) = H_q \Big( \bigoplus_{|\sigma| = p+1} C_* (X^{u-e_\sigma}) \Big) \cong \bigoplus_{|\sigma| = p+1}  H_q (X^{u-e_\sigma}) .
\]
Let us explicitly write the 1-page of the Mayer-Vietoris spectral sequence in our setting:
\begin{equation*} 
\begin{tikzcd}
\vdots  & \vdots   &   &  \vdots  \\
E^1_{0,q}= \bigoplus_{j} H_q (X^{u-e_j})  & \arrow[l,"\overline{\delta}_{1,q}",swap] E^1_{1,q}= \bigoplus_{j<h} H_q(X^{u-e_j-e_h})   & \arrow[l,"\overline{\delta}_{2,q}",swap]  \cdots & \arrow[l,"\overline{\delta}_{n-1,q}",swap] E^1_{n-1,q}= H_q (X^{u- e_{[n]}})  \\
\vdots  & \vdots   &   &  \vdots  \\
E^1_{0,1}= \bigoplus_{j} H_1 (X^{u-e_j})  & \arrow[l,"\overline{\delta}_{1,1}",swap] E^1_{1,1}=\bigoplus_{j<h} H_1 (X^{u-e_j-e_h})   & \arrow[l,"\overline{\delta}_{2,1}",swap]  \cdots & \arrow[l,"\overline{\delta}_{n-1,1}",swap] E^1_{n-1,1}= H_1 (X^{u- e_{[n]}})  \\
E^1_{0,0}= \bigoplus_{j} H_0 (X^{u-e_j})  & \arrow[l,"\overline{\delta}_{1,0}",swap] E^1_{1,0}=\bigoplus_{j<h} H_0 (X^{u-e_j-e_h})   & \arrow[l,"\overline{\delta}_{2,0}",swap]  \cdots & \arrow[l,"\overline{\delta}_{n-1,0}",swap] E^1_{n-1,0}= H_0 (X^{u- e_{[n]}})  \\
\end{tikzcd}
\end{equation*}
We display only the first quadrant $p,q\ge 0$, since elsewhere the terms $E^1_{p,q}$ are null. Moreover,
the columns of indices $p=0,1,\ldots,n-1$ we showed in the diagram are the only  (possibly) non-null ones. The differentials $d^1_{p,q}: E^1_{p,q}\to E^1_{p-1,q}$ are the maps induced in homology by the horizontal differentials $\delta_{p,q}$ of the double complex, which we denote $\overline{\delta}_{p,q}$. 
Explicitly, the differential 
\begin{equation}
\label{eq:d1_restr}
d^1_{p,q} = \overline{\delta}_{p,q} : E^1_{p,q} \cong \bigoplus_{|\sigma| = p+1}  H_q ( X^{u-e_\sigma} ) \longrightarrow E^1_{p-1,q} \cong \bigoplus_{|\tau| = p}  H_q ( X^{u-e_\tau} )
\end{equation}
is the linear map acting on each direct summand $H_q ( X^{u-e_{\sigma}} )$ of the domain by
\begin{equation*}
d^1_{p,q |} = \sum_{i=0}^p (-1)^i \iota_q^{u-e_\sigma, u-e_{\partial_{p+1-i} \sigma}} ,
\end{equation*}
where $\iota_q^{u-e_\sigma, u-e_{\partial_{\ell} \sigma}}$ denotes the map induced in $q$th homology by the inclusion $C_*(X^{u-e_\sigma})\hookrightarrow C_*(X^{u-e_{\partial_{\ell} \sigma}})$, for each for $1\le \ell \le p+1$. Let us recall that, if $\sigma =\{ j_1 < \cdots < j_{p+1} \}$, we denote  $\partial_{\ell} \sigma \=\{ j_1 < \cdots < \hat{\jmath}_{\ell} < \cdots < j_{p+1} \}$.
We observe that each row in the $1$-page $\{ E^1_{p,q}, d^1_{p,q}\}_{p,q\in \Z}$ is the truncation of a Koszul complex, for each degree $q$ of homology. More precisely, the $q$th row $\{ E^1_{p,q}, d^1_{p,q}\}_{p\in\Z}$ is a truncated version of the Koszul complex $\K_*(x_1,\ldots,x_n;V_q)(u)$, with the chain group $\K_0(x_1,\ldots,x_n;V_q)(u) = H_q(X^u)$ replaced by the zero vector space. We will prove the details of this claim in Proposition \ref{prop:MV2page}.

We obtain the $2$-page of the Mayer-Vietoris spectral sequence by taking homology of the horizontal chain complexes in the $1$-page. For our purposes, we are not as interested in its terms as we are in their dimensions (as vector spaces), which we express as follows in terms of the Betti tables $\xi_i^q(u)$, dropping in the notation the dependence on $u$ for readability's sake:  
\begin{equation*} 
\begin{tikzcd}
\vdots  & \vdots   &   &  \vdots  \\
\dim E^2_{0,q}=\dim \im \mer_q +\xi^q_1  &   \dim E^2_{1,q}=\xi^q_2   &   \cdots & \dim E^2_{n-1,q}=\xi^q_n  \\
\vdots  & \vdots   &   &  \vdots  \\
\dim E^2_{0,1}=\dim \im \mer_1 +\xi^1_1  &   \dim E^2_{1,1}=\xi^1_2   &   \cdots & \dim E^2_{n-1,1}=\xi^1_n  \\
\dim E^2_{0,0}= \dim \im \mer_0 +\xi^0_1  &   \dim E^2_{1,0}=\xi^0_2   &   \cdots & \dim E^2_{n-1,0}=\xi^0_n  \\
\end{tikzcd}
\end{equation*}
As before, we have only $n$ (possibly) non-null columns, corresponding to $p=0,\ldots ,n-1$. For $1\le p\le n-1$, it is clear why the multi-graded Betti numbers appear in the table, since they are defined as the dimension of the homology groups of the Koszul complex. 
In Proposition \ref{prop:MV2page} we prove the equalities in the column $p=0$, upon rigorously checking the claims we made regarding the $1$-page.

\begin{prop}
\label{prop:MV2page}
For each $q\ge 0$, the
$q$th row $\{ E^1_{p,q}, d^1_{p,q}\}_{p\in\Z}$ of the $1$-page of the Mayer-Vietoris spectral sequence associated with $\{ X^{u-e_j}\}_{j\in [n]}$ coincides with the truncation of the Koszul complex $\K_{*+1}(x_1,\ldots,x_n;V_q)(u)$.
The terms of the $2$-page have dimension
\begin{equation*}
\label{eq:dimE2} 
\dim E^2_{p,q}=
\left\{\begin{array}{lr}
\dim \im \mer^u_q + \xi^q_1(u) & \text{if } p = 0\\
\xi^q_{p+1}(u) & \text{if } 1 \le p \le n-1 \\
0  & \text{otherwise}
\end{array}\right. 
\end{equation*}
\end{prop}
\begin{proof}
For all $p,q\ge 0$, it is clear that $E^1_{p,q}=\bigoplus_{|\sigma| = p+1} H_q ( X^{u-e_{\sigma}})$ coincides with 
$\mathbb{K}_{p+1} (x_1,\ldots ,x_n; V_q)(u)$ as defined in (\ref{eq:koszul-space}). 
Comparing the explicit description (\ref{eq:d1_restr}) of the differentials $d^1_{p,q}=\overline{\delta}_{p,q}$ with the differentials $d_{p}$ of the Koszul complexes $\mathbb{K}_{*} (x_1,\ldots ,x_n; V_q)(u)$, defined in (\ref{eq:koszul-differential}), we observe that they coincide up to a shift in grading: $\overline{\delta}_{p,q}=d_{p+1}$, for all $p\ge 1$. 

Since $E^2_{p,q} \cong \ker d^1_{p,q}/\im d^1_{p+1,q}$ for all $p,q\ge 0$, it follows that $\dim E^2_{p,q}=\xi^q_{p+1}(u)$ for all $p\ge 1$. We remark that $\dim E^2_{p,q}=\xi^q_{p+1}(u) = 0$ if $p\ge n$. It is also clear that $\dim E^2_{p,q}=0$ if $p<0$. 
If $p=0$, as an effect of the truncation of the Koszul complex (which clearly does not affect the other columns) we have
\begin{align*}  
\dim E^2_{0,q} & = \dim \left( \mathbb{K}_{1} (x_1,\ldots ,x_n; V_q)(u) /\im d_2 \right) \\
& = \dim \left( \oplus_{j} H_q (X^{u-e_j}) /\im \overline{\delta}_{1,q} \right) \\
& = \dim ( \oplus_j H_q (X^{u-e_j})) - \dim \im \overline{\delta}_{1,q}  \\
& = \dim \im \mer^u_q + \dim \ker \mer^u_q - \dim \im \overline{\delta}_{1,q}  \\
& = \dim \im \mer^u_q + \xi_1^q(u) ,
\end{align*}
where the last two equalities follow from the existence of the differential 
\[
d_1 = \mer_q^u : \K_1(x_1,\ldots,x_n;V_q)(u) = \oplus_j H_q (X^{u-e_j}) \longrightarrow \K_0(x_1,\ldots,x_n;V_q)(u) = H_q (X^u)
\] 
in the non-truncated Koszul complex.
\end{proof}

\begin{rmk} Since $\dim \im \mer^u_q = \dim H_q (X^u) - \xi_0^q (u)$, if $p=0$ the statement of Proposition \ref{prop:MV2page} can be equivalently expressed as
$\dim E^2_{0,q} = \dim H_q (X^u) -\xi_0^q(u) +\xi_1^q(u)$.
\end{rmk}

Let us now focus on convergence and on the $\infty$-page of the Mayer-Vietoris spectral sequence. 

\begin{prop}
\label{prop:union}
Let $\{ X^u\}_{u\in \Z^n}$ be an $n$-parameter filtration. The Mayer-Vietoris spectral sequence of $\{ X^{u-e_j}\}_{j\in [n]}$, for a fixed grade $u \in \Z^n$, has $E^n_{p,q}=E^{\infty}_{p,q}$, for all $p,q$, and
\begin{equation}   
\label{eq:convergence_n}
H_k (\cup_j X^{u-e_j}) \cong \bigoplus_{p+q=k} E^n_{p,q} = \bigoplus_{i=0}^k E^{n}_{i,k-i} ,
\end{equation}
for all $k\ge 0$.
\end{prop}
\begin{proof} Since $E^r_{p,q}=0$ for all $q$ and $r$ whenever $p<0$ or $p\ge n$, for each term $E^n_{p,q}$ both the incoming differential $d^n_{p+n,q-n+1}: E^n_{p+n,q-n+1} \to E^n_{p,q}$ and the outgoing differential $d^n_{p,q}: E^n_{p,q} \to E^n_{p-n,q+n-1}$ are trivial, so $E^n_{p,q}=E^{n+1}_{p,q}=\cdots =E^{\infty}_{p,q}$. 
We saw in Section \ref{subsec:MV_preliminaries} that the spectral sequence converges to $H_*(T_*)\cong H_* (\cup_j X^{u-e_j})$. Recall that $E^{\infty}_{p,q} \cong \mathcal{F}_p H_{k}(T_*) / \mathcal{F}_{p-1} H_{k}(T_*)$, for every $p$, $q$ and $k=p+q$, where $\{ \mathcal{F}_p H_* (T_*) \}_{p\in \Z}$ is the filtration on $H_*(T_*)$ defined by (\ref{eq:inducedFpHT}). Since the spectral sequence is in the first quadrant, 
we have $\bigoplus_{i=0}^k E^n_{i,k-i}= \bigoplus_{i=0}^k E^\infty_{i,k-i}\cong H_k (T_*)$.
\end{proof}

Let us now consider diagram (\ref{eq:def_mer}) and observe that, for the Mayer-Vietoris spectral sequence associated with $\{ X^{u-e_j}\}_{j\in [n]}$, the map $\overline{\varepsilon}^u_q$ is induced in homology by the chain map $\varepsilon_*: \bigoplus_{j} C_* (X^{u-e_j}) \to  C_* (\cup_j X^{u-e_j})$ induced by the inclusions $X^{u-e_j}\hookrightarrow \cup_{j} X^{u-e_j}$.
We end this section by showing in Theorem \ref{thm:Einfty0} that the image of $\overline{\varepsilon}^u_q$ is isomorphic to the term $E^{\infty}_{0,q}$ of the spectral sequence. We first need to state a general result in homological algebra \cite[p.~165--166]{Brown1982}, a proof of which can be found, in a cohomological setting, in \cite[III.7, Lemma 12]{gelfand2003methods}.

\begin{lem}
\label{lem:Manin}
Let $C_{*,*}=\{ C_{p,q}, \delta_{p,q}, \partial_{p,q} \}$ be a first quadrant double complex with associated total complex $T_*$, let $K_*$ be a chain complex and let $\varepsilon_* : C_{0,*} \to K_*$ be a chain map. Assume that
\[
0 \xleftarrow{} K_* \xleftarrow{\varepsilon_*} C_{0,*} \xleftarrow{\delta_{1,*}}  C_{1,*} \xleftarrow{} \cdots \xleftarrow{} C_{p-1,*} \xleftarrow{\delta_{p,*}} C_{p,*} \xleftarrow{} \cdots
\]
is an exact sequence of chain complexes. Consider the induced chain map $\tilde{\varepsilon}_* : T_* \to K_*$ defined by the maps $\tilde{\varepsilon}_k : T_k = C_{0,k}\oplus \cdots \oplus C_{k,0} \to K_k$ sending $(c_0,\ldots ,c_k)$ to $\varepsilon_k (c_0)$. 
Then $\tilde{\varepsilon}_*$ induces isomorphisms 
\[ H_k (\tilde{\varepsilon}_*) : H_k (T_*) \to H_k (K_*)
\]
in homology, for each $k\in \Z$.
\end{lem}

\begin{thm}
\label{thm:Einfty0}
The terms of the $\infty$-page of the Mayer-Vietoris  spectral sequence of $\{ X^{u-e_j}\}_{j\in [n]}$ having index $p=0$ satisfy
\begin{equation}
\label{eq:Einfty0}
E^{\infty}_{0,q} \cong \im \Big( \overline{\varepsilon}^u_q :
\bigoplus_{j} H_q (X^{u-e_j}) \to  H_q (\cup_j X^{u-e_j})
\Big) .
\end{equation}
\end{thm}
\begin{proof}
The Mayer-Vietoris spectral sequence is the spectral sequence associated with the filtration $\{ F^{\text{I}}_p T_* = \bigoplus_{i\le p} C_{i,*-i} \}_{p\in \Z}$ of the total complex $T_*$ of the double complex $C_{*,*}$ introduced in (\ref{eq:Cdouble}). Consider the induced filtration $\{ \mathcal{F}_p H_* (T_*) \}_{p\in \Z}$ on $H_*(T_*)$ defined as in (\ref{eq:inducedFpHT}). Since $C_{*,*}$ is a first quadrant double complex, $\mathcal{F}_p H_k (T_*)=0$ whenever $p<0$. In particular, for $p=0$ we have 
\begin{equation}
\label{eq:Einfty_F0}
E^{\infty}_{0,q} \cong \frac{\mathcal{F}_0 H_{q}(T_*)}{\mathcal{F}_{-1} H_{q}(T_*)} = \mathcal{F}_0 H_{q}(T_*) = \im (f^0_q : H_q (F^{\text{I}}_0 T_*) \to H_q (T_*) ) .    
\end{equation}
We complete the proof by showing that $\im f^0_q \cong \im \overline{\varepsilon}^u_q$. Since $F^{\text{I}}_0 T_* = C_{0,*}$ and the chain map $\varepsilon_* : C_{0,*} \to C_* (\cup_j X^{u-e_j})$ fits into the exact sequence (\ref{eq:augm}), we can apply Lemma \ref{lem:Manin} and conclude that the induced map $\tilde{\varepsilon}_* : T_* \to C_* (\cup_j X^{u-e_j})$, which makes the triangle 
\begin{equation*}
\label{eq:comm_epsilon_bis}
\begin{tikzcd}
 C_{0,*} \arrow[dr, "\varepsilon_*", swap] \arrow[hookrightarrow]{r} & T_* \arrow[d, "\tilde{\varepsilon}_*"]  \\
& C_* (\cup_j X^{u-e_j})
\end{tikzcd}
\end{equation*}
commutative, induces isomorphisms in homology. By applying $q$th homology and observing that $H_q (C_{0,*}) = \bigoplus_j H_q (X^{u-e_j})$ we obtain the commutative triangle 
\begin{equation*}
\label{eq:comm_epsilon_homology}
\begin{tikzcd}
 \bigoplus_j H_q (X^{u-e_j}) \arrow[dr, "\overline{\varepsilon}_q^u", swap] \arrow[r, "f^0_q"] & H_q (T_*) \arrow[d, "\cong"]  \\
& H_q (\cup_j X^{u-e_j})
\end{tikzcd}
\end{equation*}
which combined with (\ref{eq:Einfty_F0}) completes the proof.
\end{proof}

\section{Morse inequalities for persistence modules}
\label{sec:strong}

Inspired by the standard Morse inequalities reviewed in Section \ref{sec:Morse-inequalities}, our goal in this section is to prove analogous inequalities for persistence modules obtained from an underlying (multi)-filtered cell complex ${\mathcal X}=\{X^u\}_{u\in \Z^n}$, $n\ge 1$: 
\begin{thm}
\label{thm:strong_relative_xi}
For each $q\ge 0$, and each fixed grade $u\in \Z^n$, we have
\[
\sum_{i=0}^q (-1)^{q+i} \mu_i (u)  \ge \sum_{i=0}^q (-1)^{q+i} \left( \xi_0^i(u) - \sum_{p=1}^{i+1}\xi_p^{i+1-p}(u)    \right) ,
\]
with $\mu_q(u)=\dim H_q(X^u, \cup_j X^{u-e_j})$, $\xi_p^q(u)=\dim H_p(\K_*(x_1,\ldots,x_n;V_q)(u))$, and $V_q=H_q({\mathcal X})$.
\end{thm}
 
In our setting, the right-hand side of the inequality involves the Betti tables of the Koszul complex $\K_*(x_1,\ldots,x_n;V_q)$ of the persistence modules $V_q$ associated with a  multi-filtration ${\mathcal X}$ of $X$. They play the same role as the Betti numbers in standard Morse inequalities. Similarly, the left-hand side involves the homological Morse numbers $\mu_q(u)$, defined (see Section \ref{subsec:multipers}) as the dimension of the $q$-th relative homology group of the pair $(X^u, \cup_j X^{u-e_j})$. 

Before proving the theorem, let us state as a consequence an analogue of the weak Morse inequalities which follow from the strong ones of Theorem \ref{thm:strong_relative_xi} in the usual way (see Section \ref{sec:Morse-inequalities}).

\begin{coro}
\label{coro:weak_relative_xi}
For each $q\ge 0$, and each fixed grade $u\in\Z^n$, we have
\[
\mu_q (u) \ge   \xi_0^{q}(u) - \sum_{p=1}^{q+1}\xi_p^{q+1-p}(u).
\]
\end{coro}

\begin{proof}
Using the inequality of Theorem \ref{thm:strong_relative_xi} for both $q$ and $q-1$ we obtain
\begin{align*}
\mu_q(u)&= \sum_{i=0}^q (-1)^{q+i} \mu_i (u) + \sum_{i=0}^{q-1} (-1)^{q-1+i} \mu_i (u) 
\\
&\ge   \sum_{i=0}^q (-1)^{q+i} \left( \xi_0^i(u) - \sum_{p=1}^{i+1}\xi_p^{i+1-p}(u)    \right)  + \sum_{i=0}^{q-1} (-1)^{q-1+i} \left( \xi_0^i(u) - \sum_{p=1}^{i+1}\xi_p^{i+1-p}(u)    \right)
\\
&=  \xi_0^q(u) - \sum_{p=1}^{q+1}\xi_p^{q+1-p}(u) .
 \end{align*}
\end{proof}

Let us now prove Theorem \ref{thm:strong_relative_xi}. First,  let us recall a simple but useful fact.

\begin{prop}
\label{prop:subadd_seq}
In an exact sequence of finite-dimensional vector spaces with a final zero
\[
A_d \xrightarrow{f_d} B_d \xrightarrow{g_d} C_d \xrightarrow{h_d} A_{d-1} \xrightarrow{f_{d-1}} \cdots \xrightarrow{h_1} A_0 \xrightarrow{f_0} B_0 \xrightarrow{g_0} C_0 \xrightarrow{h_0} 0
\]
we have
\[
\sum_{i=0}^d (-1)^{d+i} \dim A_i +
\sum_{i=0}^d (-1)^{d+i} \dim C_i
\ge  \sum_{i=0}^d (-1)^{d+i} \dim B_i  .
\]
\end{prop}

\begin{proof}
Consider the exact sequence 
\[
0\rightarrow \ker f_d \rightarrow A_d \xrightarrow{f_d} B_d \xrightarrow{g_d} C_d \xrightarrow{h_d} A_{d-1} \xrightarrow{f_{d-1}} \cdots \xrightarrow{h_1} A_0 \xrightarrow{f_0} B_0 \xrightarrow{g_0} C_0 \xrightarrow{h_0} 0 .
\]
The fact that the alternating sum of the dimensions vanishes can be expressed as
\[
\sum_{i=0}^d (-1)^{d+i} \dim A_i  +
\sum_{i=0}^d (-1)^{d+i} \dim C_i
= \dim \ker f_d + \sum_{i=0}^d (-1)^{d+i} \dim B_i ,
\]
which implies our claim.
\end{proof}

We are now ready to prove our strong Morse inequalities.

\begin{proof}[Proof of Theorem \ref{thm:strong_relative_xi}]
As usual, we denote by $\{ E^r_{p,q}, d^r_{p,q}\}$ the Mayer-Vietoris spectral sequence associated with $\{ X^{u-e_j}\}_{j\in [n]}$ for the fixed grade $u \in \Z^n$ of the statement. For each $r$, we will be interested only in the terms $E^r_{p,q}$ indexed by $(p,q)\in I_k$, where $I_k \= \{ (p,q)\in \Z^2 \mid p,q\ge 0 \text{ and } p+q \le k \}$, for a fixed $k\ge 0$. It is however convenient for the sake of bookkeeping to consider all $(p,q)\in \Z^2$ such that $p\ge 0$ and $0\le p+q\le k$, keeping in mind that $E^r_{p,q}=0$ if $q<0$. For fixed $r,p,k\ge 0$ consider the chain complex
\begin{equation}
\label{eq:truncated_Er}
\cdots \to E^{r}_{p,k-p} \xrightarrow{d^{r}_{p,k-p}} E^{r}_{p-r,k-p+r-1} \xrightarrow{d^{r}_{p-r,k-p+r-1}} \cdots \to 0 .
\end{equation}
It is clear that every such chain complex, built using the appropriate terms and differentials of the $r$-page, eventually ends with zero terms. Even if in the Mayer-Vietoris spectral sequence the chain complex (\ref{eq:truncated_Er}) may extend on the left with non-zero terms, we now consider only the portion displayed in (\ref{eq:truncated_Er}), 
restricting to terms of total degree not larger than $k$. By the standard strong Morse inequalities (\ref{eq:strongCd}) we have 
\[
\sum_{\ell \ge 0} (-1)^{\ell} \dim E^r_{p-\ell r,k-p+\ell (r-1)}  \ge \sum_{\ell \ge 0} (-1)^{\ell} \dim E^{r+1}_{p-\ell r,k-p+ \ell (r-1)} .
\]
Allowing $p\ge 0$ to vary, it is easy to observe that each term $E^r_{p',q'}$ with $(p',q')\in I_k$ appears in one (and only one) of the chain complexes (\ref{eq:truncated_Er}).
Keeping $k$ fixed, we can sum over all $p\ge 0$ and obtain 
\[
\sum_{p\ge 0}\sum_{\ell \ge 0} (-1)^{\ell} \dim E^r_{p-\ell r,k-p+\ell (r-1)}  \ge \sum_{p\ge 0}\sum_{\ell \ge 0} (-1)^{\ell} \dim E^{r+1}_{p-\ell r,k-p+\ell (r-1)} .
\]
The choice of signs in the alternating sums is such that the terms of total degree $k-i$ have sign $(-1)^i$, for each $i\ge 0$. We can therefore write this inequality as 
\begin{equation}
\label{eq:ineq_Er}
\sum_{i=0}^k (-1)^{k+i} \sum_{p+q=i} \dim E^r_{p,q}  \ge 
\sum_{i=0}^k (-1)^{k+i} \sum_{p+q=i} \dim E^{r+1}_{p,q} .
\end{equation}
Let us recall now that for the $n$-page, by convergence of the spectral sequence (see Proposition \ref{prop:union}) we have
\[
\sum_{i=0}^k (-1)^{k+i} \sum_{p+q=i} \dim E^n_{p,q}  = 
 \sum_{i=0}^k (-1)^{k+i} \dim H_i (\cup_j X^{u-e_j}) .
\]
By (repeatedly) applying (\ref{eq:ineq_Er}) we obtain 
\begin{equation}
\label{eq:thmproof_strong_doublecmp}
\sum_{i=0}^k (-1)^{k+i} \dim H_i (\cup_j X^{u-e_j}) =
 \sum_{i=0}^k (-1)^{k+i} \sum_{p+q=i} \dim E^n_{p,q} \le   \sum_{i=0}^k (-1)^{k+i} \sum_{p+q=i} \dim E^2_{p,q} , 
\end{equation}
and since by Proposition \ref{prop:MV2page} we know that $\sum_{p+q=i} \dim E^2_{p,q}= \dim H_i (X^u) -\xi_0^i (u) + \sum_{p=1}^{i+1}\xi_p^{i+1-p} (u)$,
we have
\begin{equation}
\label{eq:thmproof_strong_doublecmp2}
\sum_{i=0}^k (-1)^{k+i} \dim H_i (\cup_j X^{u-e_j}) \le   \sum_{i=0}^k (-1)^{k+i} \dim H_i (X^u) + \sum_{i=0}^k (-1)^{k+i} \left( \sum_{p=1}^{i+1}\xi_p^{i+1-p}(u)  -\xi_0^i(u)  \right) .
\end{equation}
On the other hand, applying Proposition \ref{prop:subadd_seq} to  the long exact sequence of relative homology of the pair $(X^u, \cup_j X^{u-e_j})$,
\[
H_k(\cup_j X^{u-e_j}) \xrightarrow{} H_k (X^u) \xrightarrow{} H_k (X^u, \cup_j X^{u-e_j}) \xrightarrow{} H_{k-1} (\cup_j X^{u-e_j})  \xrightarrow{} \cdots \xrightarrow{} H_0 (X^u, \cup_j X^{u-e_j}) \xrightarrow{} 0 ,
\]
yields the inequality
\begin{equation*}
\label{eq:subadd_relativehom}
\sum_{i=0}^k (-1)^{k+i} \dim H_i (\cup_j X^{u-e_j}) +
\sum_{i=0}^k (-1)^{k+i} \dim H_i (X^u, \cup_j X^{u-e_j})
\ge  \sum_{i=0}^k (-1)^{k+i} \dim H_i (X^u) ,
\end{equation*}
which combined with (\ref{eq:thmproof_strong_doublecmp2}) yields
\[
\sum_{i=0}^k (-1)^{k+i} \dim H_i (X^u, \cup_j X^{u-e_j}) \ge  - \sum_{i=0}^k (-1)^{k+i} \left( \sum_{p=1}^{i+1}\xi_p^{i+1-p}(u)  -\xi_0^i (u) \right) . 
\]
\end{proof}

\section{Euler characteristic  for persistence modules}
\label{sec:euler}

In this section we derive, using our strong Morse inequalities, Euler characteristic formulas for the relative homology of a multi-filtration ${\mathcal X}=\{ X^u \}_{u\in \Z^n}$ involving the Betti tables. 
Euler characteristic formulas are ubiquitous in homological algebra, as they are based on a general and well-known result valid for any free and bounded chain complex. For $n$-parameter persistence modules, for example, the Euler characteristic of a minimal free resolution is considered in \cite{gafvert2017stable}. 
Here, we consider instead the Euler characteristic of the chain complex $C_*(X^u, \cup_j X^{u-e_j})$ for any fixed grade $u\in \Z^n$.

Firstly, it is worth observing that, in our setting, the Euler characteristic of $C_*(X^u)$ can be expressed in terms of Betti tables as follows.

\begin{prop}
\label{prop:chi-non-relative}
For persistence modules $V_q=H_q({\mathcal X})$ obtained from an $n$-parameter filtration ${\mathcal X}=\{ X^u \}_{u\in \Z^n}$, it holds that
\[
\chi(X^u) \= \sum_{p,q\ge 0} (-1)^{p+q} \sum_{v\preceq u} \xi_{p}^{q} (v) ,
\]
with $\chi(X^u)\=\sum_{q\ge 0} (-1)^q \dim H_q (X^u)$.  
\end{prop}
\begin{proof}
Proposition 2.3 of \cite{lesnick2019computing} 
states the following relation between the point-wise dimension at $u\in \Z^n$ of a (finitely presented) $n$-parameter persistence module $V$ and its Betti tables $\xi_j$, which is an easy consequence of Hilbert's Syzygy theorem:
\[
\dim V^u = \sum_{j=0}^n (-1)^j \sum_{v\preceq u} \xi_j (v) .
\]
Applying this formula in the particular case of  a persistence module $V_q=H_q({\mathcal X})$, we  obtain 
\[
\sum_{q\ge 0} (-1)^q \dim H_q (X^u) = \sum_{p,q\ge 0} (-1)^{p+q} \sum_{v\preceq u} \xi_{p}^{q} (v) .
\]
\end{proof}

We consider now the Euler characteristic of the pair $(X^u,\cup_j X^{u-e_j})$, defined as \[
\chi (X^u,\cup_j X^{u-e_j})\=\sum_{q\ge 0} (-1)^q \dim H_q (X^u,\cup_j X^{u-e_j}),
\]
which in our notation is equal to $\sum_{q\ge 0} (-1)^q \mu_q(u)$.
We derive the following result on the Euler characteristic $\chi (X^u,\cup_j X^{u-e_j})$ as a corollary of our strong Morse inequalities (Theorem \ref{thm:strong_relative_xi}). 

\begin{thm}
\label{thm:chi_relative_betti}
Given an $n$-parameter filtration ${\mathcal X}=\{ X^u \}_{u\in \Z^n}$, for each fixed grade $u\in \Z^n$ the Euler characteristic of the pair $(X^u,\cup_j X^{u-e_j})$ is related to the Betti tables of persistent homology of ${\mathcal X}$ by
\[
\chi (X^u,\cup_j X^{u-e_j}) = \sum_{0\le p+q \le d+1} (-1)^{p+q}  \xi_{p}^{q} (u) = \sum_{p,q} (-1)^{p+q}  \xi_{p}^{q} (u) ,
\]
where $d$ is the dimension of ${\mathcal X}$. 
\end{thm}

\begin{rmk}
The first equality of the statement can be written also as $\chi (X^u,\cup_j X^{u-e_j}) = \sum_{i=0}^{d+1} (-1)^i \sum_{p=0}^i \xi_p^{i-p}(u)$.
The second equality of the statement is not a simple rewriting of the alternating sum, as the sum in the right-hand side ranges over all $p,q\in \Z$. Some possibly non-zero $\xi_p^q(u)$ are thus involved, which do not appear in the first alternating sum. Let us recall that in our setting the possibly non-zero $\xi_p^q(u)$ have indices $0\le p \le n$ and $0\le q \le d$.
\end{rmk}

\begin{rmk}
We can consider the Euler characteristic of an $n$-parameter persistence module as defined in \cite{gafvert2017stable} for the persistence module $V_q=H_q(\mathcal{X})$, for any fixed $q\ge 0$, which corresponds in our notations to $\chi(V_q)\= \sum_{p=0}^n (-1)^p \sum_{u\in \Z^n} \xi_p^q (u)$. Theorem \ref{thm:chi_relative_betti} clarifies its relation with $\chi (X^u, \cup_j X^{u-e_j})$:
\[
\sum_q (-1)^q \chi (V_q) = \sum_{u\in \Z^n} \chi (X^u, \cup_j X^{u-e_j}) .
\]
\end{rmk}  

\begin{proof}
For each fixed grade $u \in \Z^n$, it is clear that $\xi_p^q(u) =0$ whenever $q>d$. 
By the standard argument on the last strong inequality of a bounded chain complex (see Section \ref{sec:Morse-inequalities}), Theorem \ref{thm:strong_relative_xi} yields the equality 
\begin{equation}
\label{eq:chi_strong}
\sum_{i=0}^d (-1)^{d+i} \mu_i (u)  = \sum_{i=0}^d (-1)^{d+i} \left( \xi_0^i(u) - \sum_{p=0}^{i}\xi_{p+1}^{i-p}(u) \right) .
\end{equation}
The left-hand side of (\ref{eq:chi_strong}) is $ (-1)^d \chi(X^u, \cup_j X^{u-e_j})$ and the right-hand side can be rearranged as the sum
\begin{align*}
\sum_{i=0}^d (-1)^{d+i} \xi_0^i(u) - \sum_{i=0}^d (-1)^{d+i} \sum_{p=0}^{i}\xi_{p+1}^{i-p}(u) & =  \sum_{i=0}^{d+1} (-1)^{d+i} \xi_0^i(u) - \sum_{i=0}^d (-1)^{d+i} \sum_{p=0}^{i}\xi_{p+1}^{i-p}(u) \\
& = \sum_{i=0}^{d+1} (-1)^{d+i} \sum_{p=0}^{i}\xi_{p}^{i-p}(u) \\
& = \sum_{i=0}^{d+1} (-1)^{d+i} \sum_{p+q=i}\xi_{p}^{q}(u) \\
& = (-1)^d \sum_{i=0}^{d+1} (-1)^{i} \sum_{p+q=i}\xi_{p}^{q}(u), 
\end{align*}
where the first equation is obtained by subtracting $\xi_{0}^{d+1}(u) = 0$ and the following ones are obtained formally. This yields the first equality of the statement.

The second equality of the statement is obtained by repeating the proof with $m\= d+n$ in place of $d$, observing that $\xi_p^q(u) =0$ whenever $p+q>m$, as a consequence of the facts described in Section \ref{sec:setupMV}. 
\end{proof}

\section{Improving Morse inequalities}
\label{sec:improve}

In this section we improve the  weak Morse inequalities given in Corollary \ref{coro:weak_relative_xi}. Theorem \ref{thm:lower} gives a new lower bound for the homological Morse number $\mu_q (u)=\dim H_q (X^u,\cup_j X^{u-e_j})$ in terms of Betti tables. Reciprocally, Theorem \ref{thm:upper} will provide an upper bound for the homological Morse number $\mu_q(u)$. Finally, we will show that all these new inequalities are sharp.

\subsection{A new lower bound for homological Morse numbers}
We now derive Theorem \ref{thm:lower} improving the lower bound of Corollary \ref{coro:weak_relative_xi} for homological Morse numbers as a function of the Betti tables of the persistent homology modules.

Our strategy  is based again on the interplay between the long exact sequence of relative homology of $(X^u , \cup_j X^{u-e_j})$ and the Mayer-Vietoris spectral sequence. The connection between them is made via commutative triangles as in (\ref{eq:def_mer}). More precisely, we leverage Theorem \ref{thm:Einfty0}. The difference with Theorem \ref{thm:strong_relative_xi} is that we now track the Betti tables $\xi_p^q(u)$ (which appear as dimensions of the terms of the $2$-page of the spectral sequence) all the way to the $n$-page (which coincides with the $\infty$-page), to use then Theorem \ref{thm:Einfty0} and convergence of the spectral sequence to $H_* (\cup_j X^{u-e_j})$.

Before proving the lower bound inequality, let us show what we mean by ``tracking'' the Betti tables $\xi_p^q(u)$ by presenting the case of $n=3$ parameters as an example. Since the case of bifiltrations ($n=2$) is treated in \cite{landi2021discrete} using the Mayer-Vietoris long exact sequence, this represents the case  with the smallest number of parameters that requires the Mayer-Vietoris spectral sequence instead. As usual, we suppress in the notation of the spectral sequence the dependence on the fixed grade $u\in \Z^3$.

\noindent
{\bf Case $\mathbf{n=3}$.}
In this case, the $1$-page $\{ E^1_{p,q}, d^1_{p,q}\}$ of the Mayer-Vietoris spectral sequence of $\{ X^{u-e_j}\}_{j=1,2,3}$, for a fixed grade $u\in \Z^3$, consists of three non-null columns. As we said in Section \ref{sec:setupMV}, the rows correspond to  truncated Koszul complexes, for each degree of homology $q$:  
\begin{equation*} 
\begin{tikzcd}
\vdots  & \vdots   &    \vdots  \\
E^1_{0,2} =\bigoplus_{j} H_2 (X^{u-e_j})  & \arrow[l,"\overline{\delta}_{1,2}",swap] E^1_{1,2} =\bigoplus_{j<h} H_2(X^{u-e_j-e_h})   & \arrow[l,"\overline{\delta}_{2,2}",swap]  E^1_{2,2} =H_2 (X^{u-\sum_{i=1}^3 e_i})  \\
E^1_{0,1} =\bigoplus_{j} H_1 (X^{u-e_j})  & \arrow[l,"\overline{\delta}_{1,1}",swap] E^1_{1,1} =\bigoplus_{j<h} H_1 (X^{u-e_j-e_h})   & \arrow[l,"\overline{\delta}_{2,1}",swap]   E^1_{2,1} =H_1 (X^{u-\sum_{i=1}^3 e_i})  \\
E^1_{0,0} =\bigoplus_{j} H_0 (X^{u-e_j})  & \arrow[l,"\overline{\delta}_{1,0}",swap] E^1_{1,0} =\bigoplus_{j<h} H_0 (X^{u-e_j-e_h})   & \arrow[l,"\overline{\delta}_{2,0}",swap]   E^1_{2,0} =H_0 (X^{u-\sum_{i=1}^3 e_i})  \\
\end{tikzcd}
\end{equation*}
By taking homology we obtain the terms $E^2_{p,q}$ of the $2$-page. Differentials $d^2_{p,q} : E^2_{p,q} \to E^2_{p-2,q+1}$ between them are defined: 
\begin{equation*} 
\begin{tikzcd}
\vdots  & \vdots   &    \vdots  \\
E_{0,2}^2 = \coker \overline{\delta}_{1,2}  & E_{1,2}^2 =   \frac{\ker \overline{\delta}_{1,2}}{\im \overline{\delta}_{2,2}}  &  E_{2,2}^2 =   \ker \overline{\delta}_{2,2}  \\
E_{0,1}^2 = \coker \overline{\delta}_{1,1}  & E_{1,1}^2 =   \frac{\ker \overline{\delta}_{1,1}}{\im \overline{\delta}_{2,1}}  & \arrow[llu,"d^2_{2,1}" description] E_{2,1}^2 =   \ker \overline{\delta}_{2,1}  \\
E_{0,0}^2 = \coker \overline{\delta}_{1,0}  & E_{1,0}^2 =   \frac{\ker \overline{\delta}_{1,0}}{\im \overline{\delta}_{2,0}}  & \arrow[llu,"d^2_{2,0}" description] E_{2,0}^2 =   \ker \overline{\delta}_{2,0}  \\
\end{tikzcd}
\end{equation*}
By Proposition \ref{prop:MV2page}, the dimensions as vector spaces of the terms of the $2$-page are as follows: 
\begin{equation*} 
\begin{tikzcd}
\vdots  & \vdots   &  \vdots  \\
\dim E^2_{0,2}=\dim \im \mer^u_2 +\xi^2_1(u)  &   \dim E^2_{1,2}=\xi^2_2(u)   &  \dim E^2_{2,2}=\xi^2_3(u)  \\
\dim E^2_{0,1}=\dim \im \mer^u_1 +\xi^1_1(u)  &   \dim E^2_{1,1}=\xi^1_2(u)   &  \dim E^2_{2,1}=\xi^1_3(u)  \\
\dim E^2_{0,0}= \dim \im \mer^u_0 +\xi^0_1(u)  &   \dim E^2_{1,0}=\xi^0_2(u)   &  \dim E^2_{2,0}=\xi^0_3(u)  \\
\end{tikzcd}
\end{equation*}
In the $2$-page the only non-trivial differentials are of the form $d^2_{2,q}$, for $q\ge 0$, since for $p\ne 2$ either the domain or the target of the differentials $d^2_{p,q}$ is zero. The terms of the $3$-page can be expressed as:
\begin{equation*} 
\begin{tikzcd}
\vdots  & \vdots   &    \vdots  \\
E_{0,2}^3 = \coker d^2_{2,1}  & E_{1,2}^3 = E_{1,2}^2  &  E_{2,2}^3 =   \ker d^2_{2,2}  \\
E_{0,1}^3 = \coker d^2_{2,0}  & E_{1,1}^3 =  E_{1,1}^2  & E_{2,1}^3 =   \ker d^2_{2,1}  \\
E_{0,0}^3 = E_{0,0}^2  & E_{1,0}^3 =  E_{1,0}^2  & E_{2,0}^3 =   \ker d^2_{2,0}  \\
\end{tikzcd}
\end{equation*}
We observe that some terms, and in particular all terms $E^3_{1,q}$, have already stabilized at the $2$-page, meaning that taking homology with respect to differentials $d^2_{p,q}$ does not affect them. 
The dimensions of the terms of the $3$-page can be derived from the previous arguments:
\begin{equation*} 
\begin{tikzcd}
\vdots  & \vdots   &  \vdots  \\
\dim E^3_{0,2}= (\dim \im \mer^u_2 +\xi^2_1(u)) - \dim \im d^2_{2,1} &   \dim E^3_{1,2}=\xi^2_2(u)   &  \dim E^3_{2,2}=\xi^2_3(u) - \dim \im d^2_{2,2} \\
\dim E^3_{0,1}= (\dim \im \mer^u_1 +\xi^1_1(u)) - \dim \im d^2_{2,0} &   \dim E^3_{1,1}=\xi^1_2(u)   &  \dim E^3_{2,1}=\xi^1_3(u) - \dim \im d^2_{2,1} \\
\dim E^3_{0,0}= \dim \im \mer^u_0 +\xi^0_1(u)  &   \dim E^3_{1,0}=\xi^0_2(u)   &  \dim E^3_{2,0}=\xi^0_3(u) - \dim \im d^2_{2,0} \\
\end{tikzcd}
\end{equation*}
Let us recall that, for $n=3$, the $3$-page of the Mayer-Vietoris spectral sequence coincides with the $\infty$-page (Proposition \ref{prop:union}). 
We have thus kept track of the Betti tables $\xi_p^q(u)$ within the $3$-page (that is, $\infty$-page), meaning that we have found expressions for the dimensions of the terms $E^3_{p,q}$ involving the Betti tables. Below, we will detail in the general case how this can be used to derive the lower bound inequality for $\mu_q(u)$. 

\noindent
{\bf General case $\mathbf{n\ge 2}$.}
In order to generalize our argument for multi-filtrations with any number $n\ge 2$ of parameters, let us prove the following general fact: 

\begin{prop}
\label{prop:ineqEr+1}
For a spectral sequence $\{ E^r_{p,q}, d^r_{p,q}\}$ of finite dimensional vector spaces, the following statements hold for all $p,q$ and for all $r\ge 2$:
\begin{enumerate}
\item
\hfill%
$\displaystyle{\dim E^{r+1}_{p,q} = \dim E^{2}_{p,q} - \sum_{i=2}^{r} \dim \im d^{i}_{p,q} - \sum_{i=2}^{r} \dim \im d^{i}_{p+i,q-i+1};}$\hfill%
\mbox{}%
\item 
\hfill%
$\displaystyle{
\dim E^{2}_{p,q} - \sum_{i=2}^{r} \dim E^{2}_{p-i,q+i-1} - \sum_{i=2}^{r} \dim E^{2}_{p+i,q-i+1} \le \dim E^{r+1}_{p,q} \le \dim E^{2}_{p,q}.}$
\hfill\mbox{}
\end{enumerate}
\end{prop}
\begin{proof}
\textit{1}. 
For each pair of fixed indices $p,q$, at each page $r$ there are differentials
\begin{equation*}
\label{eq:diffsEr}
E^r_{p-r,q+r-1}  \xleftarrow{d^r_{p,q}}  E^r_{p,q} \xleftarrow{d^r_{p+r,q-r+1}}  E^r_{p+r,q-r+1} ,
\end{equation*}
so the dimension of $E^{r+1}_{p,q}\cong \ker d^r_{p,q} / \im d^r_{p+r,q-r+1}$ is  
\[ \dim E^{r+1}_{p,q} = \dim \ker d^r_{p,q} - \dim \im d^r_{p+r,q-r+1} = \dim E^r_{p,q} - \dim \im d^r_{p,q} - \dim \im d^r_{p+r,q-r+1} .
\]
This argument can be applied recursively to $\dim E^i_{p,q}$ for all $r\ge i >2$.

\textit{2}.
The inequality $\dim E^{r+1}_{p,q} \le \dim E^{2}_{p,q}$ follows from the fact that $E^{r+1}_{p,q}$ is a subquotient of $E^{r}_{p,q}$, for each $r$. To prove the other inequality, let us observe that,
since the dimension of the image of a linear map is upper bounded by the dimension of both the domain and the codomain, for each differential $d^i_{p,q}$ we have
\[
\dim \im d^i_{p,q} \le \dim E^i_{p-i,q+i-1}
\]
and for each differential $d^i_{p+i,q-i+1}$ we have
\[
\dim \im d^i_{p+i,q-i+1} \le \dim E^i_{p+i,q-i+1} .
\]
We can now apply to the right hand side of the equation of (\textit{1}.) the inequalities
\begin{align*}
\dim \im d^i_{p,q} & \le \dim E^i_{p-i,q+i-1} \le \dim E^2_{p-i,q+i-1} \\
\dim \im d^i_{p+i,q-i+1} & \le \dim E^i_{p+i,q-i+1} \le \dim E^2_{p+i,q-i+1},
\end{align*}
for any $2\le i\le r$. 
\end{proof}

It is worth observing that, depending on the indices $p,q,r$, several differentials and terms in Proposition \ref{prop:ineqEr+1} can be trivial in our situation.
For example, in the Mayer-Vietoris spectral sequence associated with an $n$-parameter filtration we know that $\im d^i_{p,q}$ and 
$E^2_{p-i,q+i-1}$ are zero whenever $p<i$, while $\im d^i_{p+i,q-i+1}$ and $E^2_{p+i,q-i+1}$ are zero whenever $p+i\ge n$.

Moving toward the proof of the lower bound inequality (Theorem \ref{thm:lower}), the following simple fact will be useful.

\begin{lem}
\label{lem:dim_coker_ker_i}
Consider the map $i^u_q : H_q (\cup_j X^{u-e_j}) \to H_q (X^u)$ and the commutative triangle  $\mer^u_q = i^u_q \overline{\varepsilon}_q^u$ as in  (\ref{eq:def_mer}). It holds that
\[
\dim \im i^u_q  \le \dim \im \mer^u_q + \dim H_q (\cup_j X^{u-e_j}) - \dim \im \overline{\varepsilon}_q^u .
\]
\end{lem}
\begin{proof}
Let us preliminarily observe that, given a composition $h$ of linear maps $f:U\to V$ followed by $g:V\to W$  between finite dimensional vector spaces, we have  
 \begin{equation}\label{eq:triangle}
 \dim \im g = \dim \im h + \dim V - \dim (\im f + \ker g) .
\end{equation}
 Indeed, from
 \[ h(U) = g(f(U)) \cong f(U)/ \{ x\in f(U) \mid g(x) =0 \}
 \]
we obtain $\dim \im h = \dim \im f - \dim (\im f \cap \ker g)$. Now we can sum $\dim \ker g$ to both sides of the equation, use the rank-nullity formula $\dim \ker g = \dim V - \dim \im g$ on the left-hand side, use Mayer-Vietoris' formula to express the right-hand side as $\dim (\im f + \ker g)$, and rearrange to obtain (\ref{eq:triangle}).

Applying Equation (\ref{eq:triangle}) to the commutative triangle (\ref{eq:def_mer}) yields
\begin{equation*}
\label{eq:coker_i_1}
\dim \im i^u_q = \dim \im \mer^u_q + \dim H_q (\cup_j X^{u-e_j}) - \dim (\im \overline{\varepsilon}_q^u + \ker i^u_q ).
\end{equation*}
We obtain the stated inequality by observing that  $\dim (\im \overline{\varepsilon}_q^u + \ker i^u_q ) \ge \dim \im \overline{\varepsilon}_q^u$.
\end{proof}

 \begin{thm}
\label{thm:lower}
For an $n$-parameter filtration $\{ X^u \}_{u\in \Z^n}$, for each grade $u\in \Z^n$, and for each $q\ge 0$, we have
\begin{equation*}
\label{eq:dim_Hrel_ineq}
\mu_q (u) \ge \xi^q_0 (u) + \xi^{q-1}_1 (u) - \sum_{i=1}^{n-1} \xi_{i+1}^{q-i} (u) + R, 
\end{equation*}
where
\begin{equation*} R =  \sum_{r=2}^{n-1} \left( \sum_{i=1}^{r-1} \dim \im d^r_{i,q-i} + \sum_{i=r+1}^{n-1} \dim \im d^r_{i,q-i} +  \sum_{i=1}^{n-1} \dim \im d^r_{i+r,q-i-r+1} \right)
\end{equation*}
is a non-negative integer.
\end{thm}

\begin{proof}
By standard application of the rank-nullity formula to the long exact sequence of the pair $(X^u, \cup_j X^{u-e_j})$, we know that $\dim H_q (X^u , \cup_j X^{u-e_j}) = \dim \coker i^u_q + \dim \ker i^u_{q-1}$. Hence,  we get 
\begin{equation} 
\label{eq:cq_imi}
\mu_q (u) =  (\dim H_q (X^u) -\dim \im i^u_q ) + ( \dim H_{q-1}(\cup_j X^{u-e_j}) -\dim \im i^u_{q-1} ) .
\end{equation}
On the right hand side we can apply Lemma \ref{lem:dim_coker_ker_i} to both $\dim \im i^u_q$ and $\dim \im i^u_{q-1}$. We obtain 
\begin{align} 
\begin{split}
\label{eq:cq1}
\mu_q(u) & \ge \xi^q_0 (u) - \sum_{i=0}^{n-1} \dim E^{n}_{i,q-i} + \dim E^{n}_{0,q} - \dim \im \mer^u_{q-1} + \dim E^{n}_{0,q-1} \\
& = \xi^q_0 (u) - \sum_{i=1}^{n-1} \dim E^{n}_{i,q-i}  - \dim \im \mer^u_{q-1} + \dim E^{n}_{0,q-1}
\end{split}
\end{align}
by recalling that $\im \overline{\varepsilon}_q^u \cong E^{\infty}_{0,q} \cong E^{n}_{0,q}$  and $H_q (\cup_j X^{u-e_j}) \cong E^{n}_{0,q} \oplus E^{n}_{1,q-1} \oplus \cdots \oplus E^{n}_{n-1,q-n+1}$ (Theorem \ref{thm:Einfty0} and Proposition \ref{prop:union}), together with the fact that $\xi_0^q(u) = \dim H_q (X^u) - \dim \im \mer_q^u$ (Section \ref{sec:koszul_cmp}).
We can now observe that, by Proposition \ref{prop:ineqEr+1},
\[
\dim E^{n}_{p,q} = \dim E^2_{p,q} - \sum_{r=2}^{n-1} \dim \im d^r_{p,q}  - \sum_{r=2}^{n-1} \dim \im d^r_{p+r,q-r+1} ,
\]
to express the last member of the inequality (\ref{eq:cq1}) as
\begin{multline*} \xi^q_0 (u) - \sum_{i=1}^{n-1} \left( \dim E^2_{i,q-i} - \sum_{r=2}^{n-1} \dim \im d^r_{i,q-i}  - \sum_{r=2}^{n-1} \dim \im d^r_{i+r,q-i-r+1} \right) \\
- \dim \im \mer^u_{q-1} + \left( \dim E^2_{0,q-1} - \sum_{r=2}^{n-1} \dim \im d^r_{0,q-1}  - \sum_{r=2}^{n-1} \dim \im d^r_{r,q-r}  \right) .
\end{multline*}
 Proposition \ref{prop:MV2page} states that $\dim E^2_{p,q} = \xi^q_{p+1} (u)$ when $p>0$, and $\dim E^2_{0,q} = \dim \im \mer^u_q + \xi^q_1 (u)$. Upon substitution of these terms in the previous expression we obtain 
\begin{multline*} \xi^q_0 (u) - \sum_{i=1}^{n-1} \left( \xi^{q-i}_{i+1}(u) - \sum_{r=2}^{n-1} \dim \im d^r_{i,q-i}  - \sum_{r=2}^{n-1} \dim \im d^r_{i+r,q-i-r+1} \right) \\
- \dim \im \mer^u_{q-1} + \dim \im \mer^u_{q-1} + \xi^{q-1}_1 (u) - \sum_{r=2}^{n-1} \dim \im d^r_{0,q-1}  - \sum_{r=2}^{n-1} \dim \im d^r_{r,q-r} ,
\end{multline*}
and rearranging:
\begin{multline*} \xi^q_0 (u)  + \xi^{q-1}_1 (u) - \sum_{i=1}^{n-1} \xi^{q-i}_{i+1}(u) + \sum_{i=1}^{n-1} \sum_{r=2}^{n-1} \dim \im d^r_{i,q-i} + \sum_{i=1}^{n-1} \sum_{r=2}^{n-1} \dim \im d^r_{i+r,q-i-r+1}  \\
 - \sum_{r=2}^{n-1} \dim \im d^r_{0,q-1}  - \sum_{r=2}^{n-1} \dim \im d^r_{r,q-r} .
\end{multline*}
We can now observe that $\sum_{r=2}^{n-1}\dim \im d^r_{0,q-1}=0$, since all the involved differentials target zero spaces, and that all the summands of $\sum_{r=2}^{n-1} \dim \im d^r_{r,q-r}$ cancel out with some summands of $\sum_{i=1}^{n-1} \sum_{r=2}^{n-1} \dim \im d^r_{i,q-i}$, namely those for which $i=r$. 
\end{proof}

We refer to the inequality $\mu_q (u) \ge \xi^q_0 (u) + \xi^{q-1}_1 (u) - \sum_{i=1}^{n-1} \xi_{i+1}^{q-i} (u)$ of Theorem \ref{thm:lower} as lower bound for $\mu_q (u)$ in terms of the Betti tables.

\subsection{An upper bound for homological Morse numbers}
\label{sec:upper}
 
We prove an upper bound in terms of the Betti tables for the homological Morse numbers $\mu_q(u)$ of an $n$-parameter filtration, with $n\ge 2$.

\begin{prop}
\label{prop:union2}
For an $n$-parameter filtration $\{ X^u \}_{u\in \Z^n}$, for each grade $u\in \Z^n$, and for each $q\ge 0$, we have
\begin{equation*}  \dim H_q (\cup_j X^{u-e_j}) 
\le  \dim \im \mer^u_q + \sum_{i=1}^n \xi_{i}^{q-i+1}(u) .
\end{equation*}  
\end{prop}

\begin{proof} 
By Proposition \ref{prop:union}, $H_q (\cup_j X^{u-e_j}) \cong E^{n}_{0,q} \oplus E^{n}_{1,q-1} \oplus \cdots \oplus E^{n}_{n-1,q-n+1}$. For all $p,q$ we have $\dim E^n_{p,q} \le \dim E^2_{p,q}$ (Proposition \ref{prop:ineqEr+1}).
We obtain
\[
\dim H_q (\cup_j X^{u-e_j}) 
\le \dim E^2_{0,q} + \dim E^2_{1,q-1}+\cdots + \dim E^2_{n-1,q-n+1},
\]
hence the claim is a consequence of Proposition \ref{prop:MV2page}.
\end{proof}
Keeping in mind the equality $\dim \im \mer^u_q = \dim H_q (X^u) -\xi_0^q(u)$, we note that, in the case of $n=2$ parameters, we have $E^2_{p,q}=E^\infty_{p,q}$ for all $p,q$  (Proposition \ref{prop:union}), hence Proposition \ref{prop:union2} can be stated as an equality: 
\begin{align*}  \dim H_q (X^{u-e_1}\cup X^{u-e_2}) = \dim H_q (X^u) -\xi_0^q(u) +\xi_1^q(u) + \xi^{q-1}_{2}(u).
\end{align*}

We can now prove the following upper bound:

\begin{thm}
\label{thm:upper}
For an $n$-parameter filtration $\{ X^u \}_{u\in \Z^n}$, for each grade $u\in \Z^n$, and for each $q\ge 0$, we have
\[ \mu_q(u)  \le \sum_{i=0}^n \xi_i^{q-i}(u) .
\]
\end{thm}

\begin{proof}
Since $\mer_q^u=i^u_q \overline{\varepsilon}^u_q$, as in the commutative diagram (\ref{eq:def_mer}),
we have
\begin{equation} \label{eq:bound2} \dim \im \mer_q^u \le \dim \im i^u_q ,
\end{equation}
for all $q\in \Z$ and $u\in \Z^n$. For a fixed grade $u\in \Z^n$, as a consequence of a simple argument on the long exact sequence of relative homology of the pair $(X^u, \cup_{j} X^{u-e_j})$,  we can write, as we did before in (\ref{eq:cq_imi}),
\begin{equation*} 
\mu_q (u) =  (\dim H_q (X^u) -\dim \im i^u_q ) + ( \dim H_{q-1}(\cup_j X^{u-e_j}) -\dim \im i^u_{q-1} ) .
\end{equation*}
By (\ref{eq:bound2}), the first parenthesis in the right-hand term is upper bounded by $\dim H_q (X^u) -\dim \im \mer_q^u =\xi^{q}_{0}(u)$. For the second parenthesis we use Proposition \ref{prop:union2}
to see that
\[ \dim H_{q-1} (\cup_j X^{u-e_j})  \le \dim \im \mer^u_{q-1} +\sum_{i=1}^n \xi_i^{q-i}(u) ,
\]
and (\ref{eq:bound2}) to conclude that
\[  \dim H_{q-1}(\cup_j X^{u-e_j}) -\dim \im i^u_{q-1}  \le\sum_{i=1}^n \xi_i^{q-i}(u) .
\]
Putting together the inequalities for the two parentheses we obtain the stated upper bound for $\mu_q(u)$.
\end{proof}

\subsection{Sharpness of lower and upper bounds}

In this subsection we show that, for a fixed $u\in \Z^n$ and $q\ge 0$, the lower bound 
\[
\mu_q (u) \ge \xi^q_0 (u) + \xi^{q-1}_1 (u) - \sum_{i=1}^{n-1} \xi_{i+1}^{q-i} (u)
\] 
of Theorem \ref{thm:lower} and the upper bound
\[ \mu_q(u)  \le \sum_{i=0}^n \xi_i^{q-i}(u) 
\]
of Theorem \ref{thm:upper} for $\mu_q(u)$ in terms of the Betti tables are sharp. As the previous inequalities are trivially seen to be equalities in the case when $X$ consists of only one 0-cell, we aim at showing 
that equalities can be attained in situations in which any of the involved  $\xi_i^k(u)$  is non-zero. 
The examples we provide focusing on the case of filtrations with $n=3$ parameters are general enough to be easily generalized to any number of parameters, as shown for instance in Figure \ref{fig:sharp} where the same construction is repeated for $n=1$ and $q=0$, $n=2$ and $q=1$, $n=3$ and $q=2$, and can be easily inferred for $n>3$ and $q=n-1$.

\paragraph{Lower bound.} For $n=3$ and $q=2$ we show examples of filtrations $\{ X^u\}_{u\in \Z^n}$ in which,  for a fixed $u\in \Z^n$, $\mu_2(u) = \xi^2_0(u) + \xi^{1}_1(u) - \xi^{1}_2(u) - \xi^{0}_3(u)$ holds and all $\xi_i^k(u)$ of the right-hand side are  non-zero. 
First, let us notice that taking the disjoint union of two filtered cell complexes results in adding both their homological Morse numbers $\mu_q$ and their Betti tables. It is therefore enough to provide examples of the following cases:
\begin{enumerate}[label=(\roman*)]
  \item $\mu_2(u) = \xi^2_0(u)-\xi^0_3(u)$, with $\xi^2_0(u),\xi^0_3(u)>0$, 
  \item $\mu_2(u) = \xi^2_0(u)-\xi^1_2(u)$, with $\xi^2_0(u),\xi^1_2(u)>0$,
  \item $\mu_2(u) = \xi^2_0(u)+\xi^1_1(u)$, with $\xi^2_0(u),\xi^1_1(u)>0$.
\end{enumerate}

\begin{figure}[htb]
\caption{Filtrations with $n$-parameters ($1\le n\le 3$) of an $(n-1)$-sphere for which inequalities of Theorem \ref{thm:lower} are sharp.}
\begin{center}
\begin{tabular}{cc}
\begin{minipage}{0.5\linewidth}
\begin{tabular}{c}
\fbox{\includegraphics[scale=0.5]{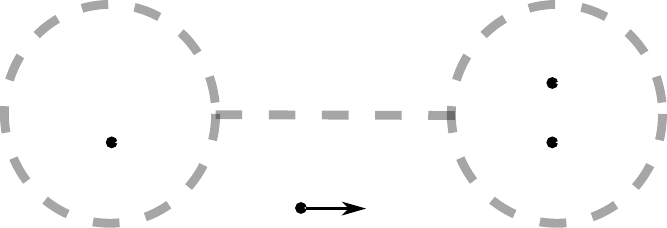}}\\
\fbox{\includegraphics[scale=0.5]{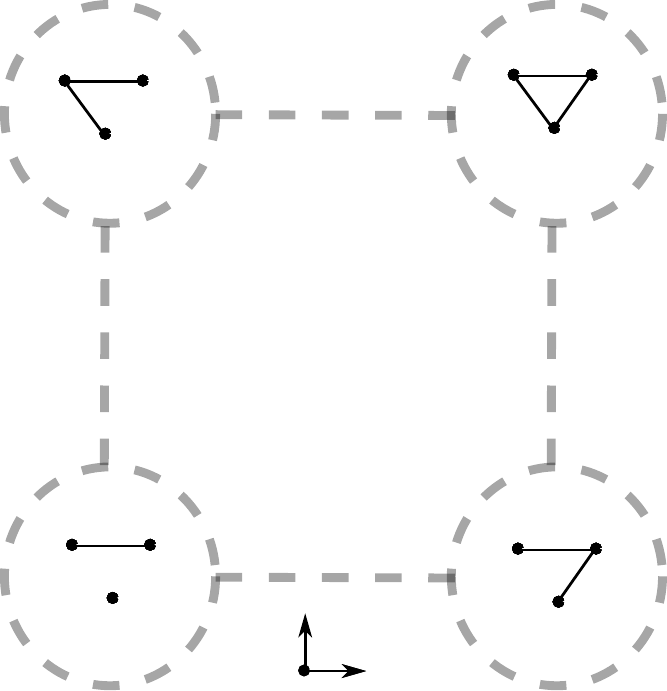}}
\end{tabular}
\end{minipage}
 &
\hspace{-1.2cm}\begin{minipage}{0.5\linewidth}
\fbox{ \includegraphics[scale=0.535]{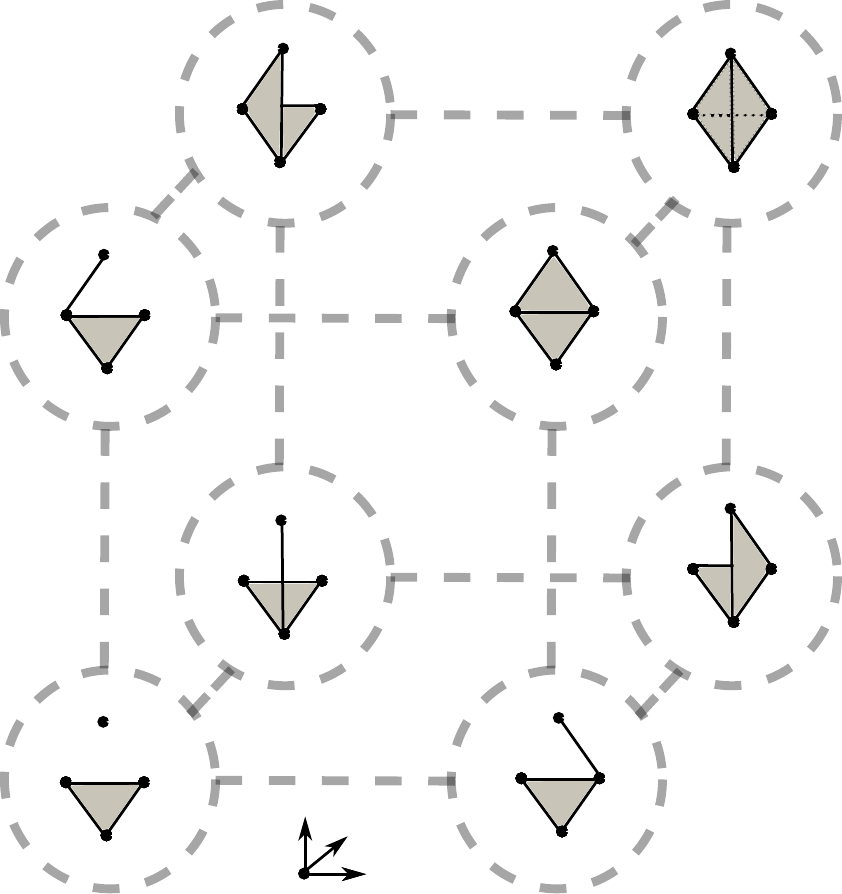}}
\end{minipage}
\end{tabular}
\end{center}
\label{fig:sharp}
\end{figure}

Cases (i) and (ii)  in particular illustrate the interesting situation of no critical cells entering at $u$ (which implies $\mu_2(u)=0$), with $\xi^2_0(u)=1$ being compensated by $\xi^0_3(u)=1$ and $\xi^1_2(u)=1$, respectively.  

For (i), consider Figure \ref{fig:sharp} (right), where $u$ is the maximum grade shown in the filtration, and only the grades $u-e_{\sigma}$ with $\sigma \subseteq \{ 1,2,3\}$ are shown. In this filtration, $X^u$ is homeomorphic to a $2$-sphere, triangulated as the boundary $\partial \Delta^3$ of a $3$-simplex. At  the minimum displayed grade $u-e_1-e_2-e_3$, only the union of the $0$-skeleton of $X^u$ and one of its $2$-faces has entered the filtration.  

For (ii), we can consider a similar filtration with at grade $u-e_1-e_2-e_3$ the union of the $1$-skeleton of $X^u=\partial \Delta^3$ and one of its $2$-faces. In this case, then, $X^{u-e_1-e_2}=X^{u-e_1-e_3}=X^{u-e_2-e_3}=X^{u-e_1-e_2-e_3}$.

Finally, for (iii), we can set $X^w=\emptyset$ at all grades $w$, except for $X^{u-e_1}\subseteq X^u$ which is the inclusion of $X^{u-e_1}\simeq S^1$ into $X^{u}\simeq S^2$ as its equator, with the entrance of $\mu_2(u)=2$ critical $2$-cells at $u$, and with $\xi^2_0(u)=\xi^1_1(u)=1$.

\paragraph{Upper bound.} For $n=3$ and $q=3$ we show examples of filtrations $\{ X^u\}_{u\in \Z^n}$ in which, for a fixed $u\in \Z^n$, $\mu_3(u) = \xi^3_0(u) + \xi^{2}_1(u) + \xi^{1}_2(u) + \xi^{0}_3(u)$ holds and all $\xi_i^k(u)$ are non-zero. 

Consider the cases (i) and (ii) we illustrated above. Adding a $3$-cell at grade $u$ so that $X^u=\Delta^3$ we have $\mu_3(u)=1$ and, respectively, $\xi^0_3(u)=1$ or $\xi^1_2(u)=1$, with the other Betti tables being zero. 

Mimicking (iii) described above, if  $X^w=\emptyset$ at all grades, except for $X^{u-e_1}\subseteq X^u$ being the inclusion of $X^{u-e_1}\simeq S^2$ into $X^{u}\simeq S^3$, then we have $\mu_3(u)=2$ and $\xi^3_0(u)=\xi^2_1(u)=1$.

\bigskip
{\bf Acknowledgments.}
We thank the anonymous referees for their constructive comments and suggestions.
This work was partially supported by the Wallenberg AI, Autonomous Systems and Software Program (WASP) funded by the Knut and Alice Wallenberg Foundation. This work was partially carried out by the last  author within the activities of ARCES (University of Bologna) and under the auspices of the GNSAGA, INdAM, Italy.


\end{document}